\documentclass[amstex,12pt, amssymb]{article}

\usepackage{mathtext}
\usepackage[dvips]{graphicx}
\usepackage{amsmath}
\usepackage{amssymb}
\usepackage{amsxtra}
\usepackage{latexsym}
\usepackage{ifthen}

\usepackage[utf8]{inputenc}

\newcounter{lemma}[section]

\newcounter{corollary}[section]

\newcounter{remark}[section]

\newcounter{theorem}[section]

\newcounter{proposition}[section]

\newcounter{definition}[section]

\newcounter{example}

\numberwithin{equation}{section}

\newtheorem{question}{Question}
\newcommand{\bques}{\begin{question}}
\newcommand{\eques}{\end{question}}
\newtheorem{comment}{Comment}
\newcommand{\bcomm}{\begin{comment}}
\newcommand{\ecomm}{\end{comment}}

\begin{document}

\markboth{\centerline{MIODRAG MATELJEVIC, RUSLAN SALIMOV, EVGENY
SEVOST'YANOV}}{\centerline{H\"{O}LDER AND LIPSCHITZ CONTINUOUS
MAPPINGS...}}

\def\cc{\setcounter{equation}{0}
\setcounter{figure}{0}\setcounter{table}{0}}

\overfullrule=0pt



\newtheorem{thm}{Theorem}
\newtheorem{lem}{Lemma}
\newtheorem{prop}{Proposition}
\newtheorem{cor}{Corollary}
\newtheorem{con}{Conjecture}
\newtheorem{rem}{Remark}
\newtheorem{df}{Definition}
\newtheorem{ex}{Example}

\def\be{\begin{equation}}
\def\ee{\end{equation}}
\newcommand{\zbar}{\overline{z}}


\author{MIODRAG MATELJEVIC, RUSLAN SALIMOV, \\ EVGENY SEVOST'YANOV}

\title{
{\bf H\"{O}LDER AND LIPSCHITZ CONTINUITY IN ORLICZ-SOBOLEV
CLASSES,  DISTORTION  AND HARMONIC MAPPINGS}}
\date{\today}
\maketitle


\begin{abstract}
In this article, we consider the H\"{o}lder continuity of  injective
maps  in  Orlicz-Sobolev classes defined on the unit ball. Under
certain conditions on the growth of dilatations, we obtain the
H\"{o}lder continuity of the indicated class of mappings. In
particular, under certain special restrictions, we show that
Lipschitz continuity of mappings holds.
We also  consider   H\"{o}lder  and  Lipschitz  continuity of harmonic  mappings  and
in particular
of  harmonic  mappings in  Orlicz-Sobolev classes.
In addition  in planar case, we  show    in some situations that the map is
bi-Lipschitzian  if Beltrami
coefficient is H\"{o}lder continuous.
\end{abstract}


\section{Introduction}
It is well known that estimates of the  distortion of  distances  of
H\"{o}lder and Lipschitz type are one of the most important objects
of modern analysis that allow a qualitative description of the
behavior of mappings. For example, obtaining such estimates can be
used to study the local behavior of solutions of the degenerate  Beltrami-type
equations, see e.g.~\cite{RSY$_1$} and
\cite{RSY$_2$}. Recall that $K$-quasiconformal mappings of the unit
disk onto itself with the normalization condition $f(0)=0$ are
H\"{o}lder continuous with exponent $1/K$ and H\"{o}lder constant $16$
(see, e.g., \cite{Ahl$_1$}, \cite[Theorem~3.2.II]{LV-73} \cite{M}
and~\cite[Theorem~18.2, Remark~18.4]{Va$_2$}). Quite simple examples
of mappings, such as the quasiconformal homeomorphism
$f=z|z|^{1/K-1},$ $f(0):=0,$ show that the H\"{o}lder exponent is
of optimal order   here  and, in particular, quasiconformal mappings, generally
speaking, are not Lipschitz. It should be noted that quasiconformal
mappings can be Lipschitz in a rather wide subclass, however, in
this case, rather specific conditions for their dilatation must be
satisfied (see, e.g., \cite{GG}). Somewhat later, similar results on
H\"{o}lder property were also established for maps with branching
(quasiregular mappings), see, for example, \cite[Theorem~3.2]{MRV}
and \cite[Theorem~1.1.2]{Re}. Subsequently, the corresponding part  of  theory of mappings
has been  developed in the direction of weakening the conditions under which
H\"{o}lder continuity or some of its analogues still holds. In
particular, the study of local estimates of the distortion  of  distances  has
long been associated with the study of mappings with finite
distortion, while H\"{o}lder continuity was often replaced by
logarithmic distance estimates, see, for example, \cite[Theorems 4
and 5]{Cr}, \cite[Theorem~11.2.3]{IM}, \cite[Theorem~7.4]{MRSY$_1$},
\cite[Theorem~3.1]{MRSY$_2$}, \cite[Theorem~5.11]{RS}
and~\cite[Theorems~1.1.V and 2.1.V]{Suv}.
\footnote{As a rule, H\"{o}lder continuity of mappings
with finite distortion does not occur, with the exception of very special
conditions for dilatation of mappings considered  in this manuscript.}
\medskip
In a number of previous publications, see,
for example, \cite{ARS}, \cite{KRSS}, \cite{KR} and
\cite{IS$_1$}--\cite{IS$_2$} we study the
Orlicz-Sobolev classe  under the Calderon
condition and this article  can be consider as continuation  of this study.
Here and below, we call the
requirement
\begin{equation}\label{eq7A}\int\limits_{1}^{\infty}\left(\frac{t}{\varphi(t)}\right)^
{\frac{1}{n-2}}dt<\infty,
\end{equation}
{\it Calderon's condition},
see~\cite{Cal}. Note that Calderon's article~\cite{Cal} containing
this requirement was published in an inaccessible journal and was
apparently forgotten. Orlicz-Sobolev classes under the Calderon
condition, denoted by  $W^{1,\varphi}_{\rm loc}$,   are in many respects more general than the mentioned
classes of quasiconformal and quasiregular mappings  and, in addition,
their properties are very similar to Sobolev classes $W_{\rm
loc}^{1, p},$ $p>n-1$, which   are obtained from the
Orlicz-Sobolev classes by choosing the function $\varphi(t)=t^p,$
$p> n-1.$  In fact,  Calderon's  condition implies   $W^{1,\varphi}_{\rm loc} \subset  W_{\rm
loc}^{1, n-1},$ $n\geq 3$.
As well as  in the Sobolev classes, functions in  the Orlicz-Sobolev classes under
the Calderon condition are differentiable almost everywhere and are
absolutely continuous on spheres with respect to the
$(n-1)$-dimensional Hausdorff measure, see, for example,
\cite[Theorem~1, Corollary~4]{KRSS}, cf.~\cite[Lemma~3]{Va$_1$},
\cite[Theorem~1.2.II]{Re}. This property enables to
establish upper bounds for the distortion of the modulus of the
families of paths (weighted Poletski inequalities) for $W_{\rm
loc}^{1, \varphi},$ and then apply the distance distortion estimates
already established previously for the corresponding mappings (see,
for example, \cite[(7.47)]{MRSY$_1$}, \cite[(4.31)]{RS}). In
particular, in our recent publication~\cite{RSS} related to the
second and third co-authors, we obtained several similar estimates
in the space and in planar   case. We also note that consideration of the
Sobolev classes  $W_{\rm loc}^{1, p}$, $p>1$, and Orlicz-Sobolev classes of this type,  in the planar  case is not required, but rather
simply Sobolev classes, see ibid. This is due, in turn, to the fact
that, by Gehring-Lehto theorem, planar  homeomorphisms of Sobolev
classes $W_{\rm loc}^{1, 1}$, without any additional restrictions,
are differentiable almost everywhere and absolutely continuous on
almost all circles with respect to the linear Lebesgue measure, and   this is  already quite enough
to establish upper bounds for the distortion of the modulus of
families of paths  (see, for example,
\cite[Theorem~3.1]{LSS} and~\cite[Theorem~3.1]{LV-73}).
The reader should be aware  that  the situation is completely different in space.  Namely, for  $n \geq 3$,  there is a homeomorphism of class  $W^{1,n-1}_{\rm loc}((-1,1)^n,\mathbb{R}^n)$  such  that  both $f$ and  $f^{-1}$ are  nowhere  differentiable.
In the survey \cite{Onof2018}  the author   clarifies  the regularity assumptions for a map to be differentiable
a.e., and  gives  some  some auxiliary results when it is not, using the notion of
approximate differentiability.  When  dealing  with  mappings  of  $W^{1,p}$  with $p < n-1$, the  notion  of
differentiability (that fails in this setting) can be replaced by the notion of approximate differentiability in the change of variable formula. However, the condition (N) plays a fundamental role for these mappings.  Indeed for such $f$, the Luzin  condition (N) is equivalent to the validity  of the area formula. If the homeomorphism $f$  satisfies the natural assumption    $f \in W^{1,n}_{\rm loc}(G,\mathbb{R}^n)$, then $f$  satisfies the condition (N). This is due to Reshetnjak\footnote{(cf. [29] in  \cite{Onof2018})}, and is a sharp result in the scale of $W^{1,p}(G,\mathbb{R}^n)$-homeomorphisms thanks to an example of Ponomarev\footnote{(cf.  [27, 28] in  \cite{Onof2018})} of a $W^{1,p}$-homeomorphisms $f: [0,1]^n\rightarrow [0,1]^n, p < n$, violating the Luzin condition (N).
\medskip
Note that   estimates of H\"{o}lder type
have   been  investigated for inner points (see,
e.g.,~\cite[Theorems~7, 8]{KRSS} and~\cite[Theorems~1.1, 4.1]{RSS}) and  that
this   article
deals with  the corresponding estimates  at
the boundary of the domain.

\medskip
The  method of moduli of families of paths   is  one of the
main research tools  (see, for example,
\cite[Corollary~9]{KRSS}, \cite[Theorem~2.2]{KR},
cf.~\cite[Theorem~7.3]{MRSY$_1$} and \cite{RS}) in the subject, and    distortion
estimates are usually proved  with    moduli techniques.

In this paper we employ the method   of the boundary extension of the
studied mappings across the boundary of a  ball using the
inversion  with respect to its  sphere, and then apply  known distortion
estimates   for the case of interior points.
So in the manuscript practically we do not use  the modulus technique directly.   Although there is a developed   theory related  to  extension theorems for Sobolev spaces
it  seems that in this context our approach  is   a novelty.
Of course, in
addition to the conditions for smoothness of mappings, this approach  also
requires analytic conditions that limit the growth of their
quasiconformality characteristics.
The article considers several similar analytic   conditions  which  are  independently from each other.

There is  a  huge  literature in the subject so it is possible that  we have  missed
to quote some important  papers for which
we apologize to the authors in advance.

\medskip
Throughout this manuscript, unless otherwise specified, $D$ denotes  a domain in ${\Bbb
R}^n,$ $n\geqslant 2. $ We  assume that the reader is familiar with  the definitions
of Sobolev classes $W_{\rm loc}^{1, 1}$ and some of their basic properties, see, for example, \cite[2.I]{Re}.
Here  only  recall  if  $f : D  \rightarrow\mathbb{R}^m$  has   ACL (absolutely continuous on lines)  property  on $D$  we write that   $f\in ACL(D)$.

We write $f\in
W^{1,\varphi}_{\rm loc}(D)$ for a locally integrable vector-function
$f=(f_1,\ldots,f_m)$ of $n$ real variables $x_1,\ldots,x_n$ if
$f_i\in W^{1,1}_{\rm loc}$ and
\begin{equation}\label{eqOS1.2a} \int\limits_{D^{\,*}}\varphi\left(|\nabla
f(x)|\right)\,dm(x)<\infty\end{equation}
for every subdomain $D^{\,*}$ with a compact closure, where $|\nabla
f(x)|=\sqrt{\sum\limits_{i,j}\left(\frac{\partial f_i}{\partial
x_j}\right)^2}.$
If additionally $f\in W^{1, 1}(D)$ and
\begin{equation}\label{eqOS1.2b} \int\limits_D\varphi\left(|\nabla
f(x)|\right)\,dm(x)<\infty\,,\end{equation}
we write $f\in W^{1,\varphi}(D).$ For a mapping $f:D\rightarrow
{\Bbb R}^n$ having partial derivatives almost everywhere in $D,$ we
set
\begin{equation}\label{eq12C}
J(x, f):=\det f^{\,\prime}(x), \quad
l\left(f^{\,\prime}(x)\right)\,=\,\,\,\min\limits_{h\in {\Bbb R}^n
\setminus \{0\}} \frac {|f^{\,\prime}(x)h|}{|h|}\,\,,
\end{equation}
for the Jacobian and smallest distortion respectively.
The inner dilatation of a  map $f$ at a  point $x\in D$ is defined
by the relation
\begin{equation}\label{eq0.1.1}
K_I(x,f)\quad =\quad\left\{
\begin{array}{rr}
\frac{|J(x,f)|}{{l\left(f^{\,\prime}(x)\right)}^n}, & J(x,f)\ne 0,\\
1,  &  f^{\,\prime}(x)=0, \\
\infty, & \text{otherwise}
\end{array}
\right.\,.
\end{equation}
In what follows, we denote by   ${\Bbb B}^n=\{x\in {\Bbb R}^n: |x|<1\},$ and
${\Bbb S}^{n-1}=\partial {\Bbb B}^n,$  respectively  the unit $n$-dimensional ball and  the unit $n-1$-dimensional  sphere.

\medskip
\begin{theorem}\label{th4} {\sl\,
Let $n\geqslant 3,$ and let $\varphi: (0, \infty)\rightarrow [0,
\infty)$ be a non-decreasing Lebesgue measurable function wich satisfies  Calderon's condition   \rm{(\ref{eq7A})}.
Suppose also that there exist constants  $C>0$ and $T>0$ such
that
\begin{equation}\label{eq30}
\varphi(2t)\leqslant C\cdot \varphi(t)\,\,\forall\,\,t\geqslant T\,.
\end{equation}
Let $Q:{\Bbb B}^n\rightarrow [0, \infty]$ be integrable function in
${\Bbb B}^n.$
Assume that $f$ is a homeomorphism of ${\Bbb B}^n$ onto ${\Bbb B}^n$
such that $f\in W^{1, \varphi}({\Bbb B}^n)$ and, in addition,
$f(0)=0.$
Let, moreover, $K_I(x, f)\leqslant Q(x)$ for a.e. $x\in {\Bbb B}^n$
and, besides that,
\begin{equation}\label{eq1AE}
\sup\limits_{\varepsilon\in(0,\varepsilon_0)}\frac{1}{\Omega_n\varepsilon^n}\int\limits_{{\Bbb
B}^n\cap B(\zeta,\varepsilon)}Q(x)\,dm(x) < C\qquad \forall\,\,
\zeta \in \partial{\Bbb B}^n
\end{equation}
holds for some $\varepsilon_0>0,$ where $\Omega_n$ is the volume of
the unit ball in ${\Bbb R}^n.$

Then $f$ has a homeomorphic extension
$f:\overline{{\Bbb B}^n}\rightarrow \overline{{\Bbb B}^n}$ and, in
addition,

$$
|f(x_2)-f(x_1)| \leqslant 2\alpha_n\varepsilon_0^{\,-\alpha}\cdot
|x_2-x_1|^{\,\alpha} \qquad \forall\,\, x_1, x_2\in \partial {\Bbb
B}^n: |x_2-x_1|<\delta_0\,,
$$
where $\delta_0:=\min\left\{\frac{1}{2},\varepsilon_0^2\right\}$,   $\omega_{n-1}$ is  the  surface of $n-1$-dimensional
sphere $\mathbb{S}^{n-1}$ and
$\alpha:=\left(\frac{\omega_{n-1}\log
2}{\Omega_n(4^n+1)2^{n+1}C}\right)^{1/(n-1)}.$
 }
\end{theorem}
Note  here that  $\omega_{n-1} = n \Omega_n$.

 In the second part of the paper we  deal with harmonic mappings and  harmonic quasiconformal (shortly hqc) mappings among the other things.
We start with  the planar case which  is very  specific. Note  that
the subject of
hqc  mappings was
essentially initiated by O.~Martio \cite{Mar} as far as we know
and then  it  has been intensively studied by the participants of
the Belgrade Analysis Seminar (see, for example,  \cite{M$_3$}
and~\cite{BM} for more details and references cited there),
and  in particular by Kalaj, who  proved that if  $h$ is a hqc
mapping of the unit disk onto  a  Lyapunov domain, then $h$ is
Lipschitz (see, e.g., \cite{Ka$_1$}). Then  in~\cite{BM} it is
proved $h$ is co-Lipschitz.  For recent development of the subject  in planar and spatial case  see \cite{mm.Spring2021}.

We also  prove    local spatial  version  of  Privalov's  theorem   for harmonic functions (Theorem  \ref{thmloc1})     which has an independent interest.
Further, we consider  harmonic maps with  the growth of distortion (dilatations) of bounded mean value.
As application of results  obtained in  in  the first  part  and  spatial  version  of  Privalov's  theorem we  show
global  H\"{o}lder continuity of mappings in considered  class.
In  particular, under certain special restrictions, we show that  Lipschitz continuity of mappings holds.

In planar case
the condition (\ref{eq11B})\footnote{see Theorem  \ref{th1}  in section \ref{s4}, p.19.} provides sufficient conditions for H\"{o}lder and Lipschitz continuity.
Next  we  show  that  in some situation if Beltrami   coefficient
is H\"{o}lder continuous that the map is biLip  \footnote{see   the section 7  which can be considered as a separate part}. As application  we
obtain  some version of  Kellogg's theorem for quasiconformal
mappings.

Finally,  let us say a few words about the activities that influenced the research contained in  this manuscript:
\begin{remark}
During Belgrade Analysis  seminar,  winter semester  2019  and  2020,  we have considered subject   related  to Geometric Function Theory (GFT) and  qc mappings  and have tried  to start some projects related to the subject\footnote{In particular,  the first author has  clarified  some facts  related to Ahlfors book \cite{Ahl$_3$}  in discussion  with V. Bo\v zin  and M. Arsenovi\'c, and   E.~Sevost'yanov  gave several lectures related  to  a ring $Q$-homeomorphism  and Orlicz-Sobolev clases. In connection with this  the first author  of this manuscript  started two independent project  a) with  M. Arsenovi\'c  related  to regularity properties of solutions of Beltrami equation,
and  b)  with  R. Salimov and E.~Sevost'yanov,  related  H\"{o}lder and Lipschitz class in Orlicz-Sobolev clas.}.
After  writing  a final version of this manuscript, in communication with  D. Kalaj  it becomes clear in particular  that  the project  is also related to some versions  of  Kellogg   and  Warschawski theorem  for a class of quasiconformal maps, and our attention has been turned first  to  \cite{Ka$_2$}  and later
to results  obtained  in \cite{MOV}  and  in  \cite{APE}; see in particular Theorem 1.3  \cite{APE}.
\end{remark}

\section{Proof of Theorem~\ref{th4}} It is known that if $f$  is $K$-qc  mapping of the unit ball  ${\Bbb B}^n$ onto itself  then  $f\in W^{1,p}({\Bbb B}^n)$ for some  $p>n$
and it satisfies hypothesis  of  Theorem~\ref{th4}  with  $\varphi(t)=t^n$.
On the other hand,  Example \ref{ex1} below
shows that  there is   a map $f\in W^{1,3}({\Bbb B}^3)$  with  unbounded inner distortion  $K_I$    satisfying the conditions of Theorem~\ref{th4}  with $Q=K_I$.
We leave the interested reader to generalize this example.

\textbf{I.} According to Corollary~6.1 in \cite{MRSY$_1$}, the
function $Q$,  which satisfies  ~(\ref{eq1AE}),  has a finite mean oscillation at each point
$x_0\in\partial{\Bbb B}^n$.
In this case, it  follows  from ~\cite[Theorem~1]{Sev$_1$}   that  there is   continuous
extension  $\tilde{f}$   of the mapping $f$ onto ${\Bbb S}^{n-1}=\partial{\Bbb B}^n$.
%
We also note that the map   $\tilde{f}$ is  a homeomorphism of the unit
ball $\overline{{\Bbb B}^n}$ onto itself, see, for example,
\cite[Lemma~6]{Sm}.

\medskip
\textbf{II.} Using conformal transformation
$\psi(x)=\frac{x}{|x|^2},$ we extend the mapping $f$
homeomorphically onto the whole $n$-dimensional Euclidean space
${\Bbb R}^n$ as follows:
\begin{equation}\label{eq1C}F(x)=\left \{\begin{array}{rr}  f(x) , & |x|<1\,,   \\
\psi(f(\psi(x))), &  |x|\geqslant 1\,.
\end{array} \right.
\end{equation}
Using  the   condition, $f\in W^{1,\varphi}({\Bbb B}^n),$  we will show
that $F\in W^{1,\varphi}(B(0, R))$ for any $R>1$ (in particular, since  $F=f$  on ${\Bbb B}^n,$ the functions  $|\nabla F|$ and
$\varphi(|\nabla F|)$ are  integrable in ${\Bbb B}^n,$ but  we show more that these functions
also integrable   in $B(0, R),$ for any $R>1$). For this, we observe that, by the
differentiation rule of a superposition of mappings,
\begin{equation}\label{eq2}F^{\,\prime}(x)=\psi^{\,\prime}(f(\psi(x))\circ
f^{\,\prime}(\psi(x))\circ \psi^{\,\prime}(x)\,.
\end{equation}
Here we used the fact that homeomorphisms of the Orlicz-Sobolev
classes under the Calderon condition are differentiable almost
everywhere, see, for example, \cite[Theorem~1]{KRSS}. As usual, put
$$\Vert f^{\,\prime}(x)\Vert\,=\,\,\,\max\limits_{h\in {\Bbb R}^n
\setminus \{0\}} \frac {|f^{\,\prime}(x)h|}{|h|}\,.$$
Using direct calculations, we may establish the inequality
\begin{equation}\label{eq1}
\Vert f^{\,\prime}(x)\Vert\leqslant|\nabla f(x)|\leqslant
n^{1/2}\cdot\Vert f^{\,\prime}(x)\Vert
\end{equation}
at all points $x\in D$ where the map $f$ has formal partial
derivatives. Observe that
$\Vert\psi^{\,\prime}(x)\Vert=\frac{1}{|x|^{2}}$ (see, e.g.,
\cite[paragraph~7]{Sev$_2$}). Recall that for two linear mappings
$g$ and $h$ the relation
\begin{equation}\label{eq3}
\Vert g\circ h\Vert\leqslant \Vert g\Vert\cdot\Vert h\Vert
\end{equation}
holds, and here, equality holds as soon as at least one of the
mappings is  orthogonal (see, e.g., \cite[I.4,
relation~(4.13)]{Re}). Since $f(0)=0,$ $f(\psi(y))\ne 0$ for
$1<|y|\leqslant R$ and for any $R>1.$ Since the map $f(\psi(y))$ is
continuous in $1\leqslant |y|\leqslant R$ and does not vanish, there
is $m>0$ such that
\begin{equation}\label{eq4}
|f(\psi(y))|\geqslant m\,,\qquad 1\leqslant |y|\leqslant R\,.
\end{equation}
In this case, from~(\ref{eq2}), (\ref{eq1}), (\ref{eq3})
and~(\ref{eq4}), we obtain that
$$
\int\limits_{1<|x|<R}|\nabla F(x)|\,dm(x)\leqslant
\int\limits_{1<|x|<R} n^{1/2}\cdot \Vert
F^{\,\prime}(x)\Vert\,dm(x)=
$$
$$=n^{1/2}\cdot \int\limits_{1<|x|<R}\Vert
\psi^{\,\prime}(f(\psi(x))\Vert\cdot \Vert
f^{\,\prime}(\psi(x))\Vert\cdot
\Vert\psi^{\,\prime}(x)\Vert\,dm(x)=$$
$$=n^{1/2}\cdot \int\limits_{1<|x|<R}
\frac{1}{|f(\psi(x))|^2}\cdot\Vert f^{\,\prime}(\psi(x))\Vert\cdot
\frac{1}{|x|^2}\,dm(x)\leqslant \frac{n^{1/2}}{m^2}\cdot
\int\limits_{1<|x|<R}\Vert f^{\,\prime}(\psi(x))\Vert\,dm(x)=$$
\begin{equation}\label{eq1A}=\frac{n^{1/2}}{m^2}\cdot
\int\limits_{1/R<|y|<1}\frac{\Vert f^{\,\prime}(y)\Vert}{|y|^{2n}}\,
dm(y)\leqslant \frac{n^{1/2}R^{\,2n}}{m^2}\cdot
\int\limits_{1/R<|y|<1}|\nabla f(y)|\, dm(y)<\infty\,.\end{equation}

\medskip
\textbf{III.} Quite similarly, applying the same arguments to the
function~$\varphi(|\nabla F|)$ instead of $|\nabla F|,$ and taking
into account relation~(\ref{eq30}) together with the non-decreasing
property of the function $\varphi,$ we obtain that
$$
\int\limits_{1<|x|<R}\varphi(|\nabla F(x)|)\,dm(x)\leqslant
C_1\cdot\int\limits_{1<|x|<R} \varphi(\Vert
F^{\,\prime}(x)\Vert)\,dm(x)=
$$
$$=C_1\cdot\int\limits_{1<|x|<R}\varphi(\Vert
\psi^{\,\prime}(f(\psi(x)))\Vert\cdot \Vert
f^{\,\prime}(\psi(x))\Vert\cdot
\Vert\psi^{\,\prime}(x)\Vert)\,dm(x)=$$
$$=C_1\cdot \int\limits_{1<|x|<R}
\varphi\left(\frac{1}{|f(\psi(x))|^2}\cdot\Vert
f^{\,\prime}(\psi(x))\Vert\cdot
\frac{1}{|x|^2}\right)\,dm(x)\leqslant
C_2\cdot\int\limits_{1<|x|<R}\varphi(\Vert
f^{\,\prime}(\psi(x))\Vert)\,dm(x)=$$
\begin{equation}\label{eq1B}=C_2\cdot
\int\limits_{1/R<|x|<1}\varphi\left(\frac{\Vert
f^{\,\prime}(x)\Vert}{|x|^{2n}}\,\right) dm(x)\leqslant
C_2R^{2n}\cdot \int\limits_{1/R<|x|<1}\varphi(|\nabla f(x)|)\,
dm(x)<\infty\,.\end{equation}
\textbf{IV.} It follows from~(\ref{eq1A}) and~(\ref{eq1B}) that
\begin{equation}\label{eq5}
\int\limits_{B(0, R)}|\nabla F(x)|\,dm(x)<\infty\,,
\int\limits_{B(0, R)}\varphi(|\nabla F(x)|)\,dm(x)<\infty\,,\quad
R>1\,.
\end{equation}
Reasoning in a similar way, we may also obtain similar relations for
the inner dilatation of the map $F.$ Indeed, since the inner
dilatation does not change under conformal mapping (see, for
example, \cite[I.4.(4.15)]{Re}), we obtain that
$$
\int\limits_{B(0, R)} K_I(x, F)\,dm(x)= \int\limits_{{\Bbb B}^n}
K_I(x, f)\,dm(x)+ \int\limits_{1<|x|<R} K_I(\psi(x), f)\,dm(x)\,.
$$
Making a change of variables here, and taking into account that
$K_I(x, f)\in L^1({\Bbb B}^n)$ by the assumption, we obtain that
$$\int\limits_{B(0, R)} K_I(x, F)\,dm(x)=\int\limits_{{\Bbb B}^n} K_I(x,
f)\,dm(x)+\int\limits_{1/R<|y|<1} K_I(y, f)\cdot
\frac{1}{|y|^{2n}}\,dm(y)\leqslant$$
\begin{equation}\label{eq12}
\leqslant \int\limits_{{\Bbb B}^n} K_I(x, f)\,dm(x)+R^{2n}\cdot
\int\limits_{1/R<|y|<1} K_I(y, f)\,dm(y)<\infty.
\end{equation}
%
\medskip
\textbf{V.} Let us check that   $F\in ACL ({\Bbb R}^n).$
It is known if $f\in
W^{1, 1}({\Bbb B}^n),$  that   the unit ball ${\Bbb B}^n$ may be divided in a standard way
into no more than a countable number of parallelepipeds $I_s$, $s\geq 1$,  with disjoint interiors,  such that   $F$
is absolutely continuous on almost all  coordinate segments  in each $I_s$,$s\geq 1$. We call a segment coordinate segment if  it is  parallel to a  coordinate axis.
Let us prove:

(A)   $F$ is absolutely
continuous on almost all segments in $\overline{{\Bbb B}^n},$
parallel to the coordinate axes.

It is enough to consider segments $r$  for which  $F$ is   absolutely
continuous (shortly AC)  on   $r_s:=r\cap I_s$   for every  $s\geq 1$.
Suppose that $r(t)=\{x\in {\Bbb R}^n: x=x_0+te, t\in [a, b]\}$ is
such a segment in $\overline{{\Bbb B}^n},$ where $e$ is some
coordinate unit vector, and $x_0\in{\Bbb B}^n.$

Two cases are
possible: when~$z_0:=x_0+be$ belongs to the interior of the ball,
and when the same point lies on the unit sphere. Set
$\alpha(t)=f(x_0+te).$ In the first case,  there are finite number of integers $s_1,s_2,...,s_l$  such that  $r=\cup_{\nu=1}^l r_{s_\nu}$. Hence  $F$ is AC on $r$.

Note also here  that by  ACL-characterization
of the Sobolev classes   (see, e.g., \cite[Theorems~1.1.2 and 1.1.3]{Ma}) and
by  the  fact  that   for a  real-valued functions defined on an interval of the real line,
absolute continuity   may be formulated by the validity of
the fundamental theorem of calculus   in terms of Lebesgue integration, (see, for example, see~\cite[Theorem~IV.7.4]{Sa}),
we have
$\int\limits_{a}^b\alpha^{\,\prime}(t)\,dt=\alpha(b)-\alpha(a)$. Let
now $z_0\in {\Bbb S}^{n-1}.$ Then, as it  was proved above with respect
to the inner points of the ball, for an arbitrary $a<c<b$ we
have that
\begin{equation}\label{eq6A}
\int\limits_{a}^c\alpha^{\,\prime}(t)\,dt=\alpha(c)-\alpha(a)\,.
\end{equation}

Since
it was also  proved above, that  the map $f$ is a homeomorphism in the closed unit
ball~$\overline{{\Bbb B}^{n}},$
the passage to the limit on the right-hand side of~(\ref{eq6A}) as
$c\rightarrow b$  gives that  $\alpha(b)-\alpha(a).$

Since (\ref{eq6A}) holds for every subinterval of $r$, we first conclude that  $F$ is AC on $r$, and (A)  follows.
Now consider  the family   $J(B(0, R))$  of  all  coordinate segments  in $B(0, R)$.
It follows from
the integrability of the gradient of the mapping $F$   on $B(0, R)$ (see ~(\ref{eq5}) and
by virtue of Fubini's theorem (see, for
example,~\cite[Theorem~III.8.1]{Sa}) that the derivative of the
function $\alpha$ is integrable on almost all segments in $B(0, R)$
parallel to the coordinate axes. Without loss of generality, we may
assume that a  segment $r(t)$ has exactly this property.

Since   the reflection with respect  to the unit sphere  is  $C^\infty$ change of variables, and   $f\in W^{1,1}({\Bbb B}^n),$
we conclude that  $F\in W^{1,1}\big( (B(0, R) \setminus {\Bbb B}^n)\big)$ (see item 1.1.7  \cite{Ma} and also  definitions  of Sobolev spaces   on manifolds in literature).
Similarly as above, we may verify that:

(B)   $F$
is absolutely continuous on almost all segments in ${\Bbb
R}^n\setminus {\Bbb B}^n,$ parallel to the coordinate axes.

Since $F$   is continuous on ${\Bbb
R}^n$, this
immediately implies that $F$ is absolutely continuous on the same
segments in ${\Bbb R}^n,$ as required.

\medskip
\textbf{VI.} Since $F\in ACL ({\Bbb R}^n),$ by~(\ref{eq5}) $F\in
W_{\rm loc}^{1, \varphi}(B(0, R))$ for any $R>1.$ Thus,
by~(\ref{eq5}) and (\ref{eq12}), $F$ is a ring $Q^{*}$-mapping in
$B(0, R),$ where $Q^*(x)=Q(x)$ for $x\in {\Bbb B}^n$ and
$Q^*(x)=Q(\psi(x))$ for $x\in B(0, R)\setminus {\Bbb B}^n$ (see,
e.g., \cite[Theorem~2.2]{KR}, cf.~\cite[Corollary~9]{KRSS}).

\medskip
\textbf{VII.} Let $\zeta_0\in {\Bbb S}^{n-1}$ and $r_0>0.$ Notice,
that
\begin{equation}\label{eq7B}
\psi(B_+(\zeta_0, \varepsilon))\subset B_-(\zeta_0,
\varepsilon)\quad \forall\,\,\varepsilon\in(0, 1)\,,
\end{equation}
where
$$B_+(\zeta_0, \varepsilon)=\{x\in {\Bbb R}^n: \exists\,e\in {\Bbb S}^{n-1},
t\in [0, \varepsilon): x=\zeta_0+te, |x|>1\}=B(\zeta_0,
\varepsilon)\cap ({\Bbb R}^n\setminus {\Bbb B}^n)\,,$$
$$B_-(\zeta_0, \varepsilon)=\{x\in {\Bbb R}^n: \exists\,e\in {\Bbb S}^{n-1},
t\in [0, \varepsilon): x=\zeta_0+te, |x|<1\}=B(\zeta_0,
\varepsilon)\cap {\Bbb B}^n\,,$$
and, as above, $\psi(x)=\frac{x}{|x|^2}.$ Indeed, for a  given
$x=\zeta_0+te\in B_+(\zeta_0, \varepsilon),$ computing the square of
the module of the vector by means  of  the scalar product $(\cdot ,
\cdot),$ we obtain that
$$|\psi(x)-\zeta_0|^{2}=\left|\frac{\zeta_0+te}{|\zeta_0+te|^2}-\zeta_0\right|^{2}=$$
$$=\frac{1}{|\zeta_0+te|^2}-\frac{2(1+t(\zeta_0, e))}{|\zeta_0+te|^2}+
\frac{|\zeta_0+te|^2}{|\zeta_0+te|^2}=$$
$$=\frac{1-2(1+t(\zeta_0, e))+1+2t(\zeta_0, e)+t^2}{|\zeta_0+te|^2}=$$
$$=\frac{t^2}{|\zeta_0+te|^2}<t^2\,,$$
that is, $|\psi(x)-\zeta_0|<t,$ as required.

\textbf{VIII.} Let $0<r<1/2.$ Now, by~(\ref{eq7B}) and by formula  for
the change of variable in the integral (see, e.g.,
\cite[Theorem~3.2.5]{Fe}) we obtain that
$$\int\limits_{B(\zeta_0, r)\cap ({\Bbb R}^n\setminus\overline{{\Bbb B}^n})}
Q^*(y)\,dm(y)=\int\limits_{B(\zeta_0, r)\cap ({\Bbb
R}^n\setminus\overline{{\Bbb B}^n})} Q(\psi(y))\,dm(y)\leqslant$$
\begin{equation}\label{eq8}\leqslant\int\limits_{B(\zeta_0, r)\cap {\Bbb B}^n} Q(y)
\cdot\frac{1}{|y|^{2n}}\,dm(y)\,.
\end{equation}
Let $y\in B(\zeta_0, r)\cap {\Bbb B}^n.$ Now $y=\zeta_0+et,$ where
$e\in {\Bbb S}^{n-1}$ and $0\leqslant t<r<1/2.$ Hence, by the
Cauchy-Bunyakovsky inequality, we have that
\begin{equation}\label{eq9}
|y|^2=|\zeta_0+et|^2=1+2t(\zeta_0, e)+t^2\geqslant
1-2t+t^2=(1-t)^2\geqslant 1/4\,.
\end{equation}
By~(\ref{eq8}) and (\ref{eq9}),
\begin{equation}\label{eq10}
\int\limits_{B(\zeta_0, r)\cap ({\Bbb R}^n\setminus\overline{{\Bbb
B}^n})} Q^*(y)\,dm(y)\leqslant 4^n\cdot \int\limits_{B(\zeta_0,
r)\cap {\Bbb B}^n} Q(y)\,dm(y)\,.
\end{equation}
It immediately follows from~(\ref{eq10}) that
\begin{equation}\label{eq11}
\int\limits_{B(\zeta_0, r)} Q^*(y)\,dm(y)\leqslant (4^n+1)\cdot
\int\limits_{B(\zeta_0, r)\cap {\Bbb B}^n} Q(y)\,dm(y)<\infty\,,
\end{equation}
because $Q$ is integrable in ${\Bbb B}^n$ by the assumption.

\medskip
\textbf{IX.} Denote by $F_Q$ the family of all homeomorphisms of the
class $W^{1,\varphi}({\Bbb B}^n)$ of the unit ball onto itself
satisfying the condition $f(0)=0,$ for which $K_I (x, f)\leqslant
Q(x)$ a.e. $x\in {\Bbb B}^n.$ According to point VI, every mapping
$F$ defined by formula~(\ref{eq1C}), where $f$ satisfies  the
hypothesis of the theorem, belongs to the class $F_Q.$  Note also
that all such mappings obviously do not take the values 0 and
$\infty$ in the domain ${\Bbb R}^n\setminus\{0\}.$
Let $h$ be a chordal metric in $\overline{{\Bbb R}^n},$
%
$$
h(x,\infty)=\frac{1}{\sqrt{1+{|x|}^2}}, \ \
h(x,y)=\frac{|x-y|}{\sqrt{1+{|x|}^2} \sqrt{1+{|y|}^2}}\,, \quad x\ne
\infty\ne y\,,
$$
%
and let $h(E):=\sup\limits_{x,y\in E}\,h(x,y)$ be a chordal diameter
of a set~$E\subset \overline{{\Bbb R}^n}$ (see, e.g.,
\cite[Definition~12.1]{Va$_2$}). Based on the above formula,
$h({\Bbb R}^n\setminus\{0\})=1.$ We note that all the statements
obtained in the proof of this theorem up to and including point VIII
hold. By~(\ref{eq11}) and (\ref{eq1AE}),
$$
\sup\limits_{\varepsilon\in(0,\varepsilon_0)}\frac{1}{\Omega_n\varepsilon^n}\int\limits_{
B(\zeta, \varepsilon)} Q^{\,*}(x)\,
dm(x)\leqslant$$$$\leqslant(4^n+1)\cdot
\sup\limits_{\varepsilon\in(0,\varepsilon_0)}\frac{1}{\Omega_n\varepsilon^n}\int\limits_{
{\Bbb B}^n\cap B(\zeta, \varepsilon) }Q(x)\, dm(x)< (4^n+1)\cdot
C\,.$$

By Lemma~3.1 in \cite{RSS} for $C_*=(4^n+1)\cdot C$ and
$\varphi(t)=1,$
$$\int\limits_{A(x_0, \varepsilon, \varepsilon_0)}
\frac{Q^{*}(x)\,dm(x)}{|x-x_0|^n} \leqslant
\frac{\Omega_n(4^n+1)2^nC}{\log
2}\left(\log\frac{1}{\varepsilon}\right),\quad\forall\,\,
\varepsilon\in (0, \varepsilon_0)\,,\quad \forall\,\,x_0\in \partial
{\Bbb B}^n\,.$$ Observe that
$ \frac{\log\frac{1}{\varepsilon}} {\log
\left(\frac{\varepsilon_0}{\varepsilon}\right)}= 1+\frac{\log
\frac{1}{\varepsilon_0}} {\log
\left(\frac{\varepsilon_0}{\varepsilon}\right)}<2 $ for
$\varepsilon\in (0, \delta_0),$ where $\delta_0>0$ is the  number defined  in
the conditions of the theorem. Now
$$\left(\log
\left(\frac{\varepsilon_0}{\varepsilon}\right)\right)^{\,-1}\cdot\int\limits_{A(x_0,
\varepsilon, \varepsilon_0)}
\frac{Q^*(x)\,dm(x)}{|x-x_0|^n}\leqslant$$
\begin{equation}\label{eq5B}
\leqslant \frac{\Omega_n(4^n+1)2^nC}{\log
2}\frac{\log\frac{1}{\varepsilon}} {\log
\left(\frac{\varepsilon_0}{\varepsilon}\right)}\leqslant
\frac{\Omega_n(4^n+1)2^{n+1}C}{\log 2}\,.
\end{equation}
Applying Lemma~4.9 in \cite{RS} for $\psi(t)=1/t$ we obtain
by~(\ref{eq5B}) that
$$h(F(x), F(x_0))\leqslant \alpha_n \left(\frac{|x-x_0|}{\varepsilon_0}\right)^{\alpha}\,,$$
for every $x\in B(x_0, \varepsilon_0)$ and any $x_0\in {\Bbb
R}^n\setminus \{0\},$ where $\alpha=\left(\frac{\omega_{n-1}\log
2}{\Omega_n(4^n+1)2^{n+1}C}\right)^{1/(n-1)}$ and $\alpha_n$ is some
constant depending only on $n.$ This inequality also remains valid
at the origin, since the same arguments and the assertion of
Lemma~4.9 in \cite{RS} apply to the map $f:{\Bbb B}^n \rightarrow
{\Bbb B}^n.$
Finally, since $h(x, y)\geqslant \frac{|x-y|}{1+r^2_0}$ for $x,y\in
\overline{B(0, r_0)},$ and $F|_{\overline{{\Bbb B}^n}}=f,$ we obtain
that
$$|f(x)-f(x_0)|\leqslant 2\alpha_n \varepsilon_0^{\,-\alpha}
|x-x_0|^{\alpha}\,.$$
Theorem is proved.~$\Box$

\medskip
\begin{example}\label{ex1}
We give an example of a map satisfying the conditions and the
conclusion of Theorem~\ref{th4} for which the function $Q$ under the
conditions of this theorem is not bounded. We consider the infinite
partition of the segment~$[0, 1]$ by points $\left[\frac{k}{k+1},
\frac{k+1}{k+2}\right],$ $k=0,1,2,\ldots \,.$ We consider the
following function $\beta:(0, 1]\rightarrow {\Bbb R},$ defined as
follows:
$$\beta(t)= \left\{
\begin{array}{rr}
1, & \left[\frac{k}{k+1}, \frac{k+1}{k+2}-2^{-4k-1}\right],\\
2^k,  &
t\in\left(\frac{k+1}{k+2}-2^{-4k-1},\,\frac{k+1}{k+2}\right)\,,
\end{array}
\right. $$ $k=0,1,2,\ldots .$
We show that for an arbitrary number $a\in [0, 1]$ the function
$\beta$ satisfies the condition
\begin{equation}\label{eq4A}
\int\limits_{a}^1\beta(t)\,dt\leqslant 2(1-a)\,.
\end{equation}
In fact, we fix such a number $a.$ Then there is $k_0\in{\Bbb N}$
such that $a\in \left(\frac{k_0}{k_0+1},
\frac{k_0+1}{k_0+2}\right].$ In this case, we obtain that
$$\int\limits_{a}^1\beta(t)\,dt=
\int\limits_{a}^{\frac{k_0+1}{k_0+2}}\beta(t)\,dt+
\int\limits_{\frac{k_0+1}{k_0+2}}^1\beta(t)\,dt =$$
$$=\frac{k_0+1}{k_0+2}-a+
\sum\limits_{k=k_0+1}^{\infty}\int\limits_{\frac{k}{k+1}}^{\frac{k+1}{k+2}-2^{-4k-1}}
\,dt+\sum\limits_{k=k_0+1}^{\infty}
\int\limits_{\frac{k+1}{k+2}-2^{-4k-1}}^{\frac{k+1}{k+2}}2^k
\,dt\leqslant$$
$$\leqslant \int\limits_a^1 dt+\sum\limits_{k=k_0+1}^{\infty}
\int\limits_{\frac{k+1}{k+2}-2^{-4k-1}}^{\frac{k+1}{k+2}}2^k
\,dt\leqslant 1-a+\sum\limits_{k=k_0+1}^{\infty} 2^{-3k-1}=$$
\begin{equation}\label{eq15}
= 1-a+\frac{2^{-3k_0-1}}{7}\leqslant 1-a+2^{-k_0-2}\,.
\end{equation}
Observe that the inequality $2^{-k_0-2}\leqslant 1-a$ is equivalent
to $a\leqslant 1-2^{-k_0-2}.$ In turn, according to the choice of
$a,$ $a\leqslant\frac{k_0+1}{k_0+2}=1-\frac{1}{k_0+2}.$ However,
$1-\frac{1}{k_0+2}\leqslant 1-2^{-k_0-2}$ is equivalent to the
obvious inequality $k_0+2\leqslant 2^{k_0+2},$ $k_0=0,1,2,\ldots .$
It follows from what has been said that
$$a\leqslant\frac{k_0+1}{k_0+2}=1-\frac{1}{k_0+2}\leqslant 1-2^{-k_0-2}\,,$$
so that by~(\ref{eq15}) we obtain that
$$\int\limits_{a}^1\beta(t)\,dt\leqslant 2(1-a)\,.$$
The relation~(\ref{eq4A}) is proved. Choosing now $\varepsilon \in
(0, 1)$ and setting $a:=1-\varepsilon,$ from (9) we obtain that
\begin{equation}\label{eq5E}
\int\limits_{1-\varepsilon}^1\beta(t)\,dt\leqslant 2\varepsilon\,.
\end{equation}
Now put $Q(x)=\beta(|x|).$ We show now that condition~(\ref{eq1AE})
is fulfilled for the indicated function~$Q.$ For simplicity, we
further consider the case~$n=3$.

\medskip
Choose an arbitrary point $\zeta_0\in {\Bbb S}^2,$ and let
$0<\varepsilon<1 $. We estimate the integral over the intersection
of the ball $B(\zeta_0, \varepsilon)$ with ${\Bbb B}^3$ using the
Fubini theorem and using some geometric considerations. Using the
formula $S=2\pi rh$ for the spherical cap  lying on the sphere of
the radius $r$ and of the hight $h,$ we may verify that
$\mathcal{H}^2(B(\zeta_0, \varepsilon)\cap {\Bbb
S}^2)=\pi\varepsilon^2,$     where   $\mathcal{H}^2$  denotes  $2$-dimensional Hausdorff measure on ${\Bbb
S}^2$.

Now, By Fubini's theorem (see, for example,
\cite[Theorem~III.8.1]{Sa}) we will have that
$$\int\limits_{B(\zeta_0, \varepsilon)\cap{\Bbb
B}^3}Q(x)\,dm(x)\leqslant\int\limits_{1-\varepsilon}^1\int\limits_{S(0,
r)\cap B(\zeta_0, \varepsilon)}\,Q(x)\,d\mathcal{H}^2\,dr=$$
$$=\int\limits_{1-\varepsilon}^1\beta(r)\int\limits_{S(0,
r)\cap B(\zeta_0, \varepsilon)}\,d\mathcal{H}^2\,dr\leqslant
\int\limits_{1-\varepsilon}^1\beta(r) \int\limits_{{\Bbb S}^2\cap
B(\zeta_0, \varepsilon)}\,d\mathcal{H}^2\,dr\leqslant
$$
\begin{equation}\label{eq7C}
\leqslant \pi\varepsilon^2
\int\limits_{1-\varepsilon}^1\beta(r)\,dr\,.
\end{equation}
It follows from~(\ref{eq5E}) and~(\ref{eq7C}) that
\begin{equation*}\label{eq1D}
\frac{3}{4\pi\varepsilon^3}\int\limits_{B(\zeta_0,
\varepsilon)\cap{\Bbb B}^3}Q(x)\,dm(x)\leqslant \frac{3}{2}\,.
\end{equation*}
Thus, for the function $Q,$ condition~(\ref{eq1AE}) is satisfied.

Guided by Proposition~6.15 in~\cite{RSY$_1$}, by analogy, we
construct the desired spatial map as follows:
$$f(x)=\frac{x}{|x|}e^{\int\limits_{1}^{|x|}(\beta(t)/t)\,dt}\,,\quad f(0):=0\,.$$
Note that the map $f,$ defined in this way, is a homeomorphism. We
verify that all the conditions of Theorem~\ref{th4} are satisfied.
Indeed, guided by Proposition~6.3 in~\cite{MRSY$_1$}, we may
calculate the tangential, radial, inner dilatations of the map $f$
and the matrix norm of $f^{\,\prime}(x)$ using the following
formulas:
$$\delta_{\tau}(x)=\frac{|f(x)|}{|x|}=e^{\int\limits_{1}^{|x|}(\beta(t)/t)\,dt}\,,
\quad \delta_r(x)=\frac{\partial|f(x)|}{\partial
|x|}=e^{\int\limits_{1}^{|x|}(\beta(t)/t)\,dt}\cdot\frac{\beta(|x|)}{|x|}\,,$$
$$\Vert f^{\,\prime}(x)\Vert=\max\{\delta_{\tau},\delta_r\}=
e^{\int\limits_{1}^{|x|}(\beta(t)/t)\,dt}\cdot\frac{\beta(|x|)}{|x|}\,,\quad
K_I(x, f)=\beta(|x|)\,.$$
Note that the norm of the map $f^{\,\prime}(x)$ is locally bounded
in ${\Bbb B}^3\setminus\{0\}$; therefore, by virtue of
inequality~(\ref{eq1}), all partial derivatives of the mapping that
exist almost everywhere are also locally bounded. From this, in
particular, it follows that the map $f$ belongs to the class $ACL$
in ${\Bbb B}^3.$

\medskip
Observe that the function $\varphi(t)=t^3$ satisfies the Calderon
condition~(\ref{eq7A}). Let us verify that the map $f$ belongs to
the class $W^{1, \varphi}({\Bbb B}^3).$ Indeed, by Fubini theorem,
$$\int\limits_{{\Bbb B}^3}\Vert f^{\,\prime}(x)\Vert^3\,dm(x)=
\int\limits_{{\Bbb
B}^3}e^{3\int\limits_{1}^{|x|}(\beta(t)/t)\,dt}\cdot\frac{\beta^3(|x|)}{|x|^3}\,dm(x)=$$
$$=\int\limits_{0<|x|<1/2}e^{3\int\limits_{1}^{|x|}(\beta(t)/t)\,dt}\cdot\frac{\beta^3(|x|)}{|x|^3}\,dm(x)
+\int\limits_{1/2<|x|<1}e^{3\int\limits_{1}^{|x|}(\beta(t)/t)\,dt}\cdot\frac{\beta^3(|x|)}{|x|^3}\,dm(x)\leqslant$$
\begin{equation}\label{eq16}\leqslant
\frac{\pi}{6}+8\pi\int\limits_{1/2}^1\beta^3(t)\,dt<\infty\,.
\end{equation}
In~(\ref{eq16}), we took into account that
$\int\limits_{1/2}^1\beta^3(t)\,dt<\infty.$ Indeed, by the
construction,
$$\int\limits_{1/2}^1\beta^3(t)\,dt\leqslant
\int\limits_{1/2}^1 dt+\sum\limits_{k=0}^{\infty}2^{-4k-1}\cdot
2^{3k}=1/2+1=3/2<\infty\,.$$
Since $f\in ACL,$ it follows from~(\ref{eq16}) that $f\in W^{1,
\varphi}({\Bbb B}^3).$  Using H\"{o}lder's inequality, it may also
be obtained from inequalities~(\ref{eq16}) that $f\in W^{1, 1}({\Bbb
B}^3).$

\medskip
We show that also $Q(x)=K_I(x, f)=\beta(|x|)\in L^1({\Bbb B}^n).$ In
fact,
$$\int\limits_{{\Bbb B}^3} K_I(x, f)\,dm(x)=4\pi\int\limits_0^1 t^2\cdot \beta(t)\,dt<\infty$$
by~(\ref{eq4A}). Thus, all the conditions of Theorem~\ref{th4} are
satisfied. Note that the map $f$ is even Lipschitz on the unit
sphere, since it identically maps some neighborhood of the unit
sphere onto itself.
\end{example}

\medskip
\begin{example}\label{ex2}
For comparison, we will also construct a similar example on the
plane. Note that this example is  related to Lemma 4.2 and Theorem 4.1
in~\cite{RSS}.

\medskip
Let $\beta$ be the function constructed in Example~\ref{ex1}. Now
put $Q(z)=\beta(|z|).$ Choose an arbitrary point $\zeta_0\in {\Bbb
S}^1,$ and let $0<\varepsilon<1 $. We are going to  estimate the integral over the
intersection of the disk $B(\zeta_0, \varepsilon)$ with ${\Bbb B}^2$
using the Fubini theorem and using some geometric considerations.
Set
$$\theta_1=\inf\limits_{z\in B(\zeta_0, \varepsilon)\cap{\Bbb
B}^2}\arg z\,,\quad\theta_2=\sup\limits_{z\in B(\zeta_0,
\varepsilon)\cap{\Bbb B}^2}\arg z\,,$$
$$L(r, \theta_1, \theta_2)=\{z\in S(0, r): z=re^{i\theta},
\theta_1<\theta<\theta_2\}\,.$$ By Fubini's theorem (see, for
example, \cite[Theorem~III.8.1]{Sa}) we will have that
$$\int\limits_{B(\zeta_0, \varepsilon)\cap{\Bbb
B}^2}Q(z)\,dm(z)\leqslant\int\limits_{1-\varepsilon}^1\int\limits_{S(0,
r)\cap B(\zeta_0, \varepsilon)}Q(z)\,|dz|dr=$$
$$=\int\limits_{1-\varepsilon}^1\beta(r)\int\limits_{L(r, \theta_1,
\theta_2)}\,|dz|dr=\int\limits_{1-\varepsilon}^1\beta(r)r(\theta_2-\theta_1)\,dr\leqslant
$$
\begin{equation}\label{eq7C1}
\leqslant
\int\limits_{1-\varepsilon}^1\beta(r)(\theta_2-\theta_1)\,dr\,.
\end{equation}
It follows from~(\ref{eq7C1}) that
\begin{equation}\label{eq1D1}
\frac{1}{\pi\varepsilon^2}\int\limits_{B(\zeta_0,
\varepsilon)\cap{\Bbb B}^2}Q(z)\,dm(z)\leqslant
\frac{1}{\pi\varepsilon^2}
\int\limits_{1-\varepsilon}^1\beta(r)(\theta_2-\theta_1)\,dr\,.
\end{equation}
Through direct geometric calculations, it can be found that
$\theta_2-\theta_1=
2\arccos\left(\frac{2-\varepsilon^2}{2}\right)\sim 2\varepsilon$ as
$\varepsilon\rightarrow 0.$ Now, $(\theta_2-\theta_1)/\varepsilon
\leqslant 3$ for sufficiently small $\varepsilon>0.$ Thus,
by~(\ref{eq5E}) and~(\ref{eq1D1})
\begin{equation}\label{eq1E}
\frac{1}{\pi\varepsilon^2}\int\limits_{B(\zeta_0,
\varepsilon)\cap{\Bbb B}^2}Q(z)\,dm(z)\leqslant
\frac{3}{\pi\varepsilon}
\int\limits_{1-\varepsilon}^1\beta(r)\,dr\leqslant
\frac{6}{\pi}:=c\,, \quad 0<\varepsilon<\varepsilon_0
\end{equation}
for some $\varepsilon_0>0.$ Thus, for the function $Q,$
condition~(\ref{eq1AE}) is satisfied. According to Proposition~6.15
in~\cite{RSY$_1$}, the mapping
$$w=f(re^{i\theta})=e^{i\theta+\int\limits_{1}^r(\beta(t)/t)\,dt}$$
is a homeomorphism of the Sobolev class $W_{\rm loc}^{1, 1}({\Bbb
C})$ and has a dilatation $K_I(z, f)$ equal to $\beta(|z|).$ Note
that this map takes the unit disk onto itself; moreover, $f|_{{\Bbb
S}^1}$ is H\"{o}lder continuous with an arbitrary exponent on ${\Bbb
S}^1$ since $f(z)=z$ in some neighborhood of ${\Bbb S}^1.$
\end{example}

\section{On H\"{o}lder type estimates for mappings with a condition on the mean value}

Given a Lebesgue measurable function $Q:{\Bbb B}^n\rightarrow [0,
\infty]$ we set $Q^*(x)=Q(x)$ for $x\in {\Bbb B}^n$ and
$Q^*(x)=Q(\psi(x))$ for $x\in B(0, R)\setminus {\Bbb B}^n,$ where
$\psi(x)=x/|x|^2.$ Put
\begin{equation}\label{eq12D*}
q_{x_0}(r):=\frac{1}{\omega_{n-1}r^{n-1}}\int\limits_{|x-x_0|=r}Q(x)\,d\mathcal{H}^{n-1}\,,
\end{equation}
\begin{equation}\label{eq12D}
q^*_{x_0}(r):=\frac{1}{\omega_{n-1}r^{n-1}}\int\limits_{|x-x_0|=r}Q^*(x)\,d\mathcal{H}^{n-1}\,,
\end{equation}
where $r>0,$ $\omega_{n-1}$ is the area of the unit sphere ${\Bbb
S}^{n-1}$ in ${\Bbb R}^n,$ and $\mathcal{H}^{n-1}$ denotes the
$(n-1)$-dimensional Hausdorff measure. When calculating the value of
$q_{x_0}(r)$ in~(\ref{eq12D*}), we assume that $Q$ is extended by
zero outside the unit ball. We prove one more important result.

\medskip
\begin{theorem}\label{th3}
{\sl\,Let $\alpha\in (0, 1].$ Suppose that under the conditions of
Theorem~\ref{th4}, instead of requirement~(\ref{eq1AE}), the
relation
\begin{equation}\label{eq11A}
\limsup\limits_{t\rightarrow
0}\int\limits_t^{\varepsilon_0}\left(\alpha-\frac{1}{q^{*\,1/(n-1)}_{x_0}(r)}\right)\cdot
\frac{dr}{r}<+\infty
\end{equation}
holds for some $x_0\in \overline{{\Bbb B}^n}$ and some
$0<\varepsilon_0<1/2,$ where $q^*_{x_0}(r)$ is defined
in~(\ref{eq12D}), and $K_I(x, f)\leqslant Q(x)$ for a.e. $x\in {\Bbb
B}^n.$ If $f$ has a homeomorphic extension $f:\overline{{\Bbb
B}^{n}}\rightarrow \overline{{\Bbb B}^{n}},$ then there exists $C>0$
and $0<\widetilde{\varepsilon_0}<\varepsilon_0$ depending only on
$n,$ $x_0$ and $Q$ such that
\begin{equation}\label{eq12B}
|f(x)-f(x_0)|\leqslant C|x-x_0|^{\alpha}\quad \forall\,\,x\in B(x_0,
\widetilde{\varepsilon_0})\cap \overline{{\Bbb B}^n}\,.
\end{equation}
 }
\end{theorem}

\begin{proof}
The proof verbatim builds points II-VI of Theorem~\ref{th4}, in
particular, the possibility of extending the mapping $f$ to the
whole space to a mapping $F$ from the same class as the original
map. Based on considerations similar to the proof of this theorem,
the map $F$ is a ring $Q^*$-map.

\medskip Denote by $F_Q$ the family of all homeomorphisms of the
class $W^{1,\varphi}({\Bbb B}^n)$ of the unit ball onto itself
satisfying the condition $f(0)=0,$ for which $K_I (x, f)\leqslant
Q(x)$ a.e. $x\in {\Bbb B}^n.$ In accordance with the above, every
mapping $F$ defined by formula~(\ref{eq1C}), where $f$ is taken from
the hypothesis of the theorem, belongs to the class $F_Q.$  Note
also that all such mappings obviously do not take the values 0 and
$\infty$ in the domain ${\Bbb R}^n\setminus\{0\}.$
Let $h$ be a chordal metric in $\overline{{\Bbb R}^n},$
%
$$
h(x,\infty)=\frac{1}{\sqrt{1+{|x|}^2}}, \ \
h(x,y)=\frac{|x-y|}{\sqrt{1+{|x|}^2} \sqrt{1+{|y|}^2}}\,, \quad x\ne
\infty\ne y\,,
$$
%
and let $h(E):=\sup\limits_{x,y\in E}\,h(x,y)$ be a chordal diameter
of a set~$E\subset \overline{{\Bbb R}^n}$ (see, e.g.,
\cite[Definition~12.1]{Va$_2$}). Based on the above formula,
$h({\Bbb R}^n\setminus\{0\})=1.$  Let $\varepsilon_0>0$ be the
number from the hypothesis of the theorem, then
by~\cite[Theorem~4.16]{RS} we have the estimate
\begin{equation}\label{eq2.4.3A}
 h(F(x),F(x_0))\leqslant \alpha_n\cdot
\exp\left\{-\int\limits_{|x-x_0|}^{\varepsilon_0}
\frac{dr}{rq_{x_0}^{\,*\,\frac{1}{n-1}}(r)}\right\}\,,
\end{equation}
for $x\in B(x_0, \varepsilon_0)$ and any $x_0\in {\Bbb
R}^n\setminus\{0\},$ where $\alpha_n$ depends only on $n.$ Note that
a similar estimate holds for the point $x_0=0,$ since the same
reasoning is applicable to the mapping $f$ in the unit ball ${\Bbb
B}^n,$ in addition, $h(\overline{{\Bbb R}^n}\setminus{\Bbb B}^n)=1.$
Since $h(x, y)\geqslant \frac{|x-y|}{1+r^2_0}$ for $x,y\in
\overline{B(0, r_0)},$ and $F|_{\overline{{\Bbb B}^n}}=f,$ it
follows from~(\ref{eq2.4.3A}) that
\begin{equation}\label{eq13}
|f(x)-f(x_0)|\leqslant 2\alpha_n\cdot
\exp\left\{-\int\limits_{|x-x_0|}^{\varepsilon_0}
\frac{dr}{rq_{x_0}^{\,*\,\frac{1}{n-1}}(r)}\right\}
\end{equation}
where $\alpha_n$ depends only on $n.$
Observe that
$$\frac{\exp\left\{-\int\limits_{t}^{\varepsilon_0}
\frac{dr}{rq_{x_0}^{\frac{1}{n-1}}(r)}\right\}}{{t^{\alpha}}}=$$
\begin{equation}\label{eq14A}
\frac{\exp\left\{-\int\limits_{t}^{\varepsilon_0}
\frac{dr}{rq_{x_0}^{\frac{1}{n-1}}(r)}\right\}}
{{\exp\left\{-\alpha\int\limits_{t}^1\frac{dr}{r}\right\}}} =
\exp\left\{\int\limits_{t}^{\varepsilon_0} \frac{\alpha
dr}{r}-\int\limits_{t}^{\varepsilon_0}\frac{dr}{rq_{x_0}^{\frac{1}{n-1}}(r)}\right\}=
\end{equation}
$$
=\exp\left\{\int\limits_{t}^{\varepsilon_0}\left(\alpha-\frac{1}{q^{1/(n-1)}_{x_0}(r)}\right)\cdot
\frac{dr}{r}\right\}\,.$$
Dividing the left side of~(\ref{eq13}) by $|x-x_0|^{\alpha}$ and
taking into account~(\ref{eq14A}), we obtain that
\begin{equation}\label{eq5A}\frac{|f(x)-f(x_0)|}{|x-x_0|^{\alpha}}\leqslant
\widetilde{C_n}\cdot
\exp\left\{\int\limits_{|x-x_0|}^{\varepsilon_0}\left(\alpha-\frac{1}{q^{\,*\,1/(n-1)}_{x_0}(r)}\right)\cdot
\frac{dr}{r}\right\}\,,
\end{equation}
where $\widetilde{C_n}=2\alpha_n.$

By~(\ref{eq11A}) there exists~$M_0>0,$ depending only on~$n,$
$\alpha$ and~$Q$ such that
\begin{equation}\label{eq13B}
\exp\left\{\int\limits_{|x-x_0|}^{\varepsilon_0}\left(\alpha-\frac{1}{q^{\,*\,1/(n-1)}_{x_0}(r)}\right)\cdot
\frac{dr}{r}\right\}\leqslant M_0\qquad \forall\,\,x\in B(x_0,
\widetilde{\varepsilon_0})\setminus\{x_0\}
\end{equation}
for some $0<\widetilde{\varepsilon_0}<\varepsilon_0.$ Now,
by~(\ref{eq5A}) and~(\ref{eq13B}) we obtain that
$$|f(x)-f(x_0)|\leqslant C\cdot|x-x_0|^{\alpha}\quad \forall\,\, x\in B(x_0, \widetilde{\varepsilon_0})\,,$$
where $C=\widetilde{C_n}\cdot M_0.$ Theorem is proved.~$\Box$
\end{proof}

\medskip
\begin{corollary}\label{cor1}
{\sl Let $\alpha\in (0, 1]$ and let $\varphi: (0, \infty)\rightarrow
[0, \infty)$ be a non-decreasing Lebesgue measurable function
with~(\ref{eq7A}). Let $Q:{\Bbb B}^n\rightarrow [0, \infty]$ be
integrable function in ${\Bbb B}^n.$ Suppose also that there exist
$C>0$ and $T>0$ such that~(\ref{eq30}) holds. Assume that $f$ is a
homeomorphism of ${\Bbb B}^n$ onto ${\Bbb B}^n$ such that $f\in
W^{1, \varphi}({\Bbb B}^n)$ and, in addition, $f(0)=0.$
Let, moreover, $K_I(x, f)\leqslant Q(x)$ for a.e. $x\in {\Bbb B}^n$
and, besides that, the relation~(\ref{eq11A}) holds for any $x_0\in
\overline{{\Bbb B}^n},$ some $0<\varepsilon_0<1/2,$ where
$q^*_{x_0}(r)$ is defined in~(\ref{eq12D}). Then $f$ has a
homeomorphic extension $f:\overline{{\Bbb B}^{n}}\rightarrow
\overline{{\Bbb B}^{n}}.$ Moreover, for any $x_0\in \overline{{\Bbb
B}^n}$ there exists $C>0$ and
$0<\widetilde{\varepsilon_0}<\varepsilon_0$ depending only on $n,$
$x_0$ and $Q$ such that~(\ref{eq12B}) holds. }
\end{corollary}

\medskip
\begin{proof}
Taking into account~(\ref{eq11A}) and~(\ref{eq14A}), we obtain that
\begin{equation}\label{eq14}
\frac{\exp\left\{-\int\limits_{t}^{\varepsilon_0}
\frac{dr}{rq_{x_0}^{\frac{1}{n-1}}(r)}\right\}}{{t^{\alpha}}}
\leqslant
\exp\left\{\int\limits_{t}^{\varepsilon_0}\left(\alpha-\frac{1}{q^{*\,1/(n-1)}_{x_0}(r)}\right)\cdot
\frac{dr}{r}\right\}\,.
\end{equation}

It follows from~(\ref{eq14}) that $\int\limits_{t}^{\varepsilon_0}
\frac{dr}{rq_{x_0}^{\frac{1}{n-1}}(r)}\rightarrow\infty$ as
$t\rightarrow 0$ for any $x_0\in\overline{{\Bbb B}^n}.$ Indeed, the
function $\alpha(t):=\exp\left\{-\int\limits_{t}^{\varepsilon_0}
\frac{dr}{rq_{x_0}^{\frac{1}{n-1}}(r)}\right\}$ is monotone by $t,$
thus, it has a limit as $t\rightarrow 0.$ Assume that
$\alpha(t)\rightarrow A$ and $A\ne 0,$ now
$$\frac{\exp\left\{-\int\limits_{t}^{\varepsilon_0}
\frac{dr}{rq_{x_0}^{\frac{1}{n-1}}(r)}\right\}}{{t^{\alpha}}}\rightarrow\infty$$
as $t\rightarrow 0.$ Now, by~(\ref{eq14}) we obtain that
$\exp\left\{\int\limits_{t}^{\varepsilon_0}\left(\alpha-\frac{1}{q^{*\,1/(n-1)}_{x_0}(r)}\right)\cdot
\frac{dr}{r}\right\}\rightarrow \infty$ as $t\rightarrow 0,$ that
contradicts~(\ref{eq11A}). The contradiction obtained above prove
that $\alpha(t)\rightarrow 0,$ which implies
$\int\limits_{t}^{\varepsilon_0}
\frac{dr}{rq_{x_0}^{\frac{1}{n-1}}(r)}\rightarrow\infty$ as
$t\rightarrow 0,$ that is desired conclusion.

\medskip
By~\cite[Theorem~2.2]{KR} the map $f$ is a ring $Q$-map at each
point $x_0\in \overline{{\Bbb B}^n}$ for $Q=K_I(x, f).$ In this
case, the possibility of homeomorphic extension of the mapping $f$
onto ${\Bbb S}^{n-1}=\partial{\Bbb B}^n$ follows
by~\cite[Theorem~1]{Sev$_1$}.
We also note that the map $f$ extends to a homeomorphism of the unit
ball $\overline{{\Bbb B}^n}$ onto itself, see, for example,
\cite[Lemma~6]{Sm}. In this case, the desired conclusion follows
from Theorem~\ref{th3}.~$\Box$
\end{proof}

\medskip
\begin{example}\label{ex4}
We give an example of a mapping whose characteristic satisfies
condition~(\ref{eq11A}) with $\alpha=1.$ Of course, an arbitrary
conformal mapping is such, since its inner and outher dilatations
are equal to 1, in addition, the map itself belongs to the class
$W^{1, n}_{\rm loc}$ and, therefore, also belongs to the class
$W^{1, \varphi}_{\rm loc}$ for example, for $\varphi=t^p,$ $p>n-1$
(see e.g.~\cite[3.I]{Re}). Since examples of such mappings are
elementary, we will not dwell on them. We give an example of a map
corresponding to condition~(\ref{eq11A}) with $\alpha=1,$ the
dilatations of which are not bounded in a neighborhood of the point
under consideration. For this purpose, we use the idea used in the
construction of Example~\ref{ex1}. We restrict ourselves to the case
$n=3.$

We consider the infinite partition of the segment~$[0, 1]$ by points
$\left[\frac{1}{k+1}, \frac{1}{k}\right],$ $k=1,2,3,\ldots \,.$ We
consider the following function $\beta:(0, 1]\rightarrow {\Bbb R},$
defined as follows:
\begin{equation}\label{eq2E}\beta(t)= \left\{
\begin{array}{rr}
1, & \left[\frac{1}{k+1}, \frac{1}{k}-2^{-4k-1}\right]\\
2^{k-1},  & t\in\left(\frac{1}{k}-2^{-4k-1}, \frac{1}{k}\right)\,,
\end{array}
\right.
\end{equation}
$k=1,2,3,\ldots .$
Setting $Q(x)=\beta(|x|),$ we obtain by~(\ref{eq12D*}) that
$q_0(r)=q^{\,*}_0(r)=\beta(r).$ Set $\varepsilon_0:=1/4$ and
$0<a<1.$ Let $k_0\in{\Bbb N}$ be a number such that $a\in
\left(\frac{1}{k_0+1}, \frac{1}{k_0}\right].$ Let us verify the
fulfillment of condition~(\ref{eq11A}) for $\alpha=1$ and
$\varepsilon_0=1/4$ at $x_0=0.$ We obtain that
$$\int\limits_{a}^{\varepsilon_0}
\frac{dr}{rq_0^{1/2}(r)}=
\sum\limits_{k=4}^{k_0-1}\int\limits_{\frac{1}{k+1}}^{\frac{1}{k}-2^{-4k-1}}\,\frac{dr}{r}+
\sum\limits_{k=4}^{k_0-1}\int\limits_{\frac{1}{k}-2^{-4k-1}}^{\frac{1}{k}}\frac{2^{(1-k)/2}\,dr}{r}
+ \int\limits_a^{\frac{1}{k_0}}\frac{dr}{r\beta^{1/2}(r)}\geqslant$$
\begin{equation}\label{eq2A}
\geqslant\sum\limits_{k=4}^{k_0-1}\int\limits_{\frac{1}{k+1}}^{\frac{1}{k}-2^{-4k-1}}\,\frac{dr}{r}\,.
\end{equation}
Observe that
$$\sum\limits_{k=4}^{k_0-1}\int\limits_{\frac{1}{k+1}}^{\frac{1}{k}-2^{-4k-1}}\,\frac{dr}{r}=
\sum\limits_{k=4}^{k_0-1}\int\limits_{\frac{1}{k+1}}^{\frac{1}{k}}\,\frac{dr}{r}-
\sum\limits_{k=4}^{k_0-1}\int\limits_{\frac{1}{k}-2^{-4k-1}}^{\frac{1}{k}}\,\frac{dr}{r}=$$
\begin{equation}\label{eq5F}
=\ln\frac{k_0}{4}-\sum\limits_{k=4}^{k_0-1}\int\limits_{\frac{1}{k}-2^{-4k-1}}^{\frac{1}{k}}\,\frac{dr}{r}\geqslant
\ln\frac{k_0}{4}-\sum\limits_{k=4}^{k_0-1}\frac{2^{-4k-1}}{\frac{1}{k}-2^{-4k-1}}\geqslant
\ln\frac{k_0}{4}-c\end{equation}
for some $0<c<\infty$ because the series
$\sum\limits_{k=4}^{\infty}\frac{2^{-4k-1}}{\frac{1}{k}-2^{-4k-1}}=
\sum\limits_{k=4}^{\infty}\frac{k\cdot 2^{-4k-1}}{1-k\cdot
2^{-4k-1}}$ converges, for example, on the Cauchy principle.  Thus,
by~(\ref{eq2A}),
\begin{equation}\label{eq3U}
\int\limits_{a}^{\varepsilon_0} \frac{dr}{rq_0^{1/2}(r)}\geqslant
\ln\frac{k_0}{4}-c\,.
\end{equation}
By~(\ref{eq3U}),
\begin{equation}\label{eq3V}
\exp\left\{-\int\limits_{a}^{\varepsilon_0}
\frac{dr}{rq_0^{1/2}(r)}\right\}\leqslant e^{\,c}\cdot
(4/k_0)=4e^{\,c}\cdot \frac{k_0+1}{k_0(k_0+1)}\leqslant
\frac{8e^{\,c}}{k_0+1}\leqslant 8e^{\,c}\cdot a\,,
\end{equation}
because $a\geqslant 1/(k_0+1)$ by the choice of $a.$ Thus,
by~(\ref{eq3V})
\begin{equation}\label{eq1V}
\lim\limits_{a\rightarrow +0}
\frac{\exp\left\{-\int\limits_{a}^{\varepsilon_0}
\frac{dr}{rq_0^{1/2}(r)}\right\}}{a}\leqslant 8e^{\,c}<\infty\,,
\end{equation}
as required. Finally, the fulfillment of relation~(\ref{eq11A})
follows on the basis of~(\ref{eq14}) and (\ref{eq1V}). It should
also be noted that condition~(\ref{eq11A}) is satisfied in the
neighborhood of any point of the unit sphere, since in some of its
neighborhood the map $f$ is identical, and the corresponding
function $Q$ is 1.

\medskip
Guided by Proposition~6.15 in~\cite{RSY$_1$}, by analogy, we
construct the desired spatial map as follows:
$$f(x)=\frac{x}{|x|}e^{\int\limits_{1}^{|x|}(\beta(t)/t)\,dt}\,,\quad f(0):=0\,.$$
Note that the map $f,$ defined in this way, is a homeomorphism. We
verify that all the conditions of Theorem~\ref{th3} are satisfied.
Indeed, by~\cite[Proposition~6.3]{MRSY$_1$}, we may calculate the
tangential, radial, inner dilatations of the map $f$ and the matrix
norm of $f^{\,\prime}(x)$ using the following formulas:
$$\delta_{\tau}(x)=\frac{|f(x)|}{|x|}=e^{\int\limits_{1}^{|x|}(\beta(t)/t)\,dt}\,,
\quad \delta_r(x)=\frac{\partial|f(x)|}{\partial
|x|}=e^{\int\limits_{1}^{|x|}(\beta(t)/t)\,dt}\cdot\frac{\beta(|x|)}{|x|}\,,$$
$$\Vert f^{\,\prime}(x)\Vert=\max\{\delta_{\tau},\delta_r\}=
e^{\int\limits_{1}^{|x|}(\beta(t)/t)\,dt}\cdot\frac{\beta(|x|)}{|x|}\,,\quad
K_I(x, f)=\beta(|x|)\,.$$
Note that the norm of the map $f^{\,\prime}(x)$ is locally bounded
in ${\Bbb B}^3\setminus\{0\}$; therefore, by virtue of
inequality~(\ref{eq1}), all partial derivatives of the mapping that
exist almost everywhere are also locally bounded. From this, in
particular, it follows that the map $f$ belongs to the class $ACL$
in ${\Bbb B}^3.$

\medskip
We now note that the function $\varphi(t)=t^3$ satisfies the
Calderon condition~(\ref{eq7A}). Let us verify that the map $f$
belongs to the class $W^{1, \varphi}({\Bbb B}^3).$ Indeed, by Fubini
theorem,
$$\int\limits_{{\Bbb B}^3}\Vert
f^{\,\prime}(x)\Vert^3\,dm(x)=\int\limits_{{\Bbb
B}^3}e^{3\int\limits_{1}^{|x|}(\beta(t)/t)\,dt}\cdot\frac{\beta^3(|x|)}{|x|^3}\,dm(x)=
$$
$$=4\pi\int\limits_0^1
e^{3\int\limits_{1}^{r}(\beta(t)/t)\,dt}\cdot\frac{\beta^3(r)}{r}\,dr
=4\pi\sum\limits_{k=1}^{\infty}
\int\limits_{\frac{1}{k+1}}^{\frac{1}{k}-2^{-4k-1}}\,r^2\,dr+$$$$+
4\pi\sum\limits_{k=1}^{\infty}\int\limits_{\frac{1}{k}-2^{-4k-1}}^{\frac{1}{k}}
e^{3\int\limits_{1}^{r}(\beta(t)/t)\,dt}\cdot\frac{\beta^3(r)}{r}\,dr\leqslant$$
$$\leqslant 4\pi\int\limits_0^1r^2\,dr+4\pi\sum\limits_{k=1}^{\infty}\int\limits_{\frac{1}{k}-2^{-4k-1}}^{\frac{1}{k}}
\frac{2^{3k-3}}{r}\,dr\leqslant$$
\begin{equation}\label{eq16A}
\leqslant(4/3)\pi+4\pi
\sum\limits_{k=1}^{\infty}\frac{2^{3k-3}}{\frac{1}{k}-2^{-4k-1}}\cdot
2^{-4k-1}<\infty\,.
\end{equation}
Since $f\in ACL,$ it follows from~(\ref{eq16A}) that $f\in W^{1,
\varphi}({\Bbb B}^3).$

\medskip
We show that also $Q(x)=K_I(x, f)=\beta(|x|)\in L^1({\Bbb B}^n).$ In
fact,
$$\int\limits_{{\Bbb B}^3} K_I(x, f)\,dm(x)=4\pi\int\limits_0^1 r^2\cdot \beta(r)\,dr=$$
$$=4\pi\sum\limits_{k=1}^{\infty}
\int\limits_{\frac{1}{k+1}}^{\frac{1}{k}-2^{-4k-1}}\,r^2\,dr+
4\pi\sum\limits_{k=1}^{\infty}\int\limits_{\frac{1}{k}-2^{-4k-1}}^{\frac{1}{k}}
r^2\cdot 2^{k-1}\,dr\leqslant$$
$$\leqslant (4\pi)/3+ 4\pi\sum\limits_{k=1}^{\infty}2^{-4k-1} \cdot 2^{k-1}<\infty\,.$$
Thus, all the conditions of Theorem~\ref{th3} are satisfied. Note
that the map $f$ is even Lipschitz on the unit sphere, since it
identically maps in some neighborhood of it. According to this
theorem, the map $f$ is Lipschitz at the point 0, and also on the
boundary of the unit ball.
\end{example}

\section{An extended version of Theorem~\ref{th3} in planar case}\label{s4}

On the plane, Theorem~\ref{th3} looks somewhat simpler; in
particular, in approximately the same classes of mappings, the
Calderon condition~(\ref{eq7A}) is not required. To state an
analogue of this theorem, we introduce the following notations.

Let $D$ be a domain in ${\Bbb C}.$ In what follows, a mapping
$f:D\rightarrow{\Bbb C}$ is assumed to be {\it sense-preserving.}
The following result is fairly close to~\cite[Theorem~1.1]{RSS},
although here, in contrast to~\cite{RSS}, the corresponding property
is established  at the boundary rather than the inner point of the
domain.

\medskip
\begin{theorem}\label{th1}{\sl\,
Let $\alpha\in (0, 1],$ and let $f$ be a homeomorphism of ${\Bbb
B}^2$ onto ${\Bbb B}^2$ such that $f\in W_{\rm loc}^{1, 1}({\Bbb
B}^2)$ and $f(0)=0.$
Let, moreover,
\begin{equation}\label{eq11B}
\limsup\limits_{t\rightarrow
0}\int\limits_t^{\varepsilon_0}\left(\alpha-\frac{1}{q^*_{z_0}(r)}\right)\cdot
\frac{dr}{r}<+\infty
\end{equation}
for some $0<\varepsilon_0<1/2$ and some $z_0\in \overline{{\Bbb
B}^2},$ where $q^*_{z_0}(r)$ is defined for $Q\in L^1({\Bbb B}^2)$
in~(\ref{eq12D}), and $K_I(z, f)\leqslant Q(z)$ a.e. in ${\Bbb
B}^2.$ If $f$ has a homeomorphic extension $f:\overline{{\Bbb
B}^2}\rightarrow \overline{{\Bbb B}^2},$ then there is $C>0$ and
$0<\widetilde{\varepsilon_0}<\varepsilon_0$ depending only on $z_0$
and $Q$ such that
\begin{equation}\label{eq12E}
|f(z)-f(z_0)|\leqslant C|z-z_0|^{\alpha}\quad \forall\,\,z\in B(z_0,
\widetilde{\varepsilon_0})\cap\overline{{\Bbb B}^2}\,.
\end{equation}}
\end{theorem}

\begin{proof} \textbf{I.} Using conformal transformation
$\psi(z)=\frac{z}{|z|^2},$ we extend the mapping $f$
homeomorphically onto ${\Bbb C}$ as follows:
$$F(z)=\left \{\begin{array}{rr}  f(z) , & |z|<1\,,   \\
\psi(f(\psi(z))), &  |z|\geqslant 1\,.
\end{array} \right.$$
As usual, put
$$\Vert f^{\,\prime}(z)\Vert\,=\,\,\,\max\limits_{h\in {\Bbb C}
\setminus \{0\}} \frac {|f^{\,\prime}(z)h|}{|h|}\,.$$
By condition, $f\in W_{\rm loc}^{1, 1}({\Bbb B}^2),$ therefore also
$F\in W_{\rm loc}^{1, 1}({\Bbb B}^2).$ We show more, namely, that
$F\in W_{\rm loc}^{1, 1}(B(0, R))$ for any $R>1$ (in particular,
$\Vert F^{\,\prime}(z)\Vert$ is not only locally integrable in
${\Bbb B}^2,$ but also globally integrable).

\medskip
\textbf{II.} Since $\Vert f^{\,\prime}(z)\Vert^2=K_I(z, f)\cdot
|J(z, f)|$ a.e., and, in addition, the inner dilatation of $f$ does
not change under conformal mappings (see, e.g.,
\cite[I.4.(4.15)]{Re}), by the H\"{o}lder inequality we obtain that
\begin{equation}\label{eq28}
\int\limits_{{{\Bbb B}^2}} \,\Vert
F^{\,\prime}(z)\Vert\,dm(z)\leqslant
 \left(\int\limits_{{\Bbb B}^2} \,
 K_I(z, F)\,dm(z)\right)^{\frac{1}{2}}\cdot
\left(\int\limits_{{\Bbb B}^2}\,|J(z,
F)|\,dm(z)\right)^{\frac{1}{2}} \,.
\end{equation}
Since the map $F$ is a homeomorphism, by~\cite[Theorems~3.1.4, 3.1.8
and 3.2.5]{Fe} we obtain that
\begin{equation}\label{eq29}
\int\limits_{{\Bbb B}^2} \,J(z, F)\,dm(z)\leqslant m(F({\Bbb
B}^2))=\pi\,.
\end{equation}
Since $K_I(z, F)\in L^1({{\Bbb B}^2})$, it follows from~(\ref{eq28})
and~(\ref{eq29}) that
\begin{equation}\label{eq310}
\int\limits_{{\Bbb B}^2} \, \Vert F^{\,\prime}(z)\Vert\,dm(z)
\leqslant \left(\pi\int\limits_{{\Bbb B}^2} \, K_I(z,
f)\,dm(z)\right)^{\frac{1}{2}}<\infty\,.
\end{equation}
Reasoning in a similar way, we may also obtain similar relations for
the inner dilatation of the map $F.$ Indeed, by~(\ref{eq2})
and~(\ref{eq3}), for any $R>1$ we obtain that
$$
\int\limits_{B(0, R)} \Vert F^{\,\prime}(z)\Vert\,dm(z)=
\int\limits_{{\Bbb B}^2} \Vert f^{\,\prime}(z)\Vert\,dm(z)+
\int\limits_{1<|z|<R} (f(\psi(z)))^{-2}\cdot\Vert
F^{\,\prime}(\psi(z))\Vert\cdot |\psi(z)|^{\,-2} \,dm(z)\leqslant
$$
$$\leqslant \int\limits_{{\Bbb B}^2} \Vert f^{\,\prime}(z)\Vert\,dm(z)+
C\cdot \int\limits_{1<|z|<R} \Vert F^{\,\prime}(\psi(z))\Vert
\,dm(z)$$
for some $C>0.$
Making a change of variables here, and taking into account that
$K_I(z, f)\in L^1({\Bbb B}^2),$ we obtain that
$$\int\limits_{B(0, R)}\Vert F^{\,\prime}(z)\Vert\,dm(z)\leqslant
\int\limits_{{\Bbb B}^2} \Vert
f^{\,\prime}(z)\Vert\,dm(z)+C\int\limits_{1/R<|y|<1} \Vert
f^{\,\prime}(y)\Vert\cdot \frac{1}{|y|^{4}}\,dm(y)\leqslant$$
\begin{equation}\label{eq12A}
\leqslant \int\limits_{{\Bbb B}^2} \Vert
F^{\,\prime}(z)\Vert\,dm(z)+CR^{4}\cdot \int\limits_{1/R<|y|<1}
\Vert f^{\,\prime}(y)\Vert\,dm(y)<\infty.
\end{equation}

\textbf{III.} By virtue of Fubini's theorem and by~ (see, for
example,~\cite[Theorem~III.8.1]{Sa}) that the derivative of the
function $\varphi=f(z_0+te),$ $t\in [a, b],$ $e\in{\Bbb S}^1,$  is
integrable on almost all segments in $B(0, R)$ parallel to the
coordinate axes. In this case, arguing in a similar way to the proof
of item~V of Theorem~\ref{th4}, we may show that $F\in ACL ({\Bbb
C}).$

\medskip
\textbf{IV.} Since $F\in ACL ({\Bbb C}),$ by~(\ref{eq5}) $F\in
W_{\rm loc}^{1, 1}(B(0, R))$ for any $R>1.$ In this case, $F$ is a
ring $Q^*$-map at each point $z_0 \in B(0, R)$ with $Q^*=Q$ in
${\Bbb B}^2$ and $Q^*(y)=Q\left(\frac{y}{|y|^2}\right)$ otherwise
(see, e.g., \cite[Theorem~3.1]{LSS}, cf.~\cite[Theorem~3.1]{Sal}).

\medskip
\textbf{V.} In view of item IV and using
again~\cite[Theorem~4.16]{RS}, we may show that
\begin{equation}\label{eq5D}\frac{|f(z)-f(z_0)|}{|z-z_0|^{\alpha}}\leqslant
\widetilde{C}\cdot
\exp\left\{\int\limits_{|z-z_0|}^{\varepsilon_0}\left(\alpha-\frac{1}{q^{\,*\,}_{z_0}(r)}\right)\cdot
\frac{dr}{r}\right\}\,.
\end{equation}
By~(\ref{eq11B}), there exists~$M_0>0,$ depending only on~$\alpha$
and~$Q$ such that
\begin{equation}\label{eq13C}
\exp\left\{\int\limits_{|z-z_0|}^{\varepsilon_0}\left(\alpha-\frac{1}{q^{\,*\,}_{z_0}(r)}\right)\cdot
\frac{dr}{r}\right\}\leqslant M_0\qquad \forall\,\,z\in B(z_0,
\widetilde{\varepsilon_0})\setminus\{z_0\}\,.
\end{equation}
Combining~(\ref{eq5D}) and (\ref{eq13C}), we arrive at the desired
relation~(\ref{eq12E}) with $C=\widetilde{C}\cdot M_0.$~$\Box$
\end{proof}

\medskip
In particular, Theorem~\ref{th1} implies the following statement.

\medskip
\begin{corollary}\label{cor2}{\sl\,
Let $\alpha\in (0, 1],$ let $z_0\in {\Bbb S}^1=\partial {\Bbb B}^2,$
and let $f$ be a homeomorphism of ${\Bbb B}^2$ onto ${\Bbb B}^2$
such that $f\in W_{\rm loc}^{1, 1}({\Bbb B}^2)$ and $f(0)=0.$
Let, moreover, (\ref{eq11B}) holds for some $0<\varepsilon_0<1/2$
and any $z_0\in \overline{{\Bbb B}^2},$ where $q^*_{z_0}(r)$ is
defined for $Q\in L^1({\Bbb B}^2)$ in~(\ref{eq12D}), and $K_I(z,
f)\leqslant Q(z)$ a.e. in ${\Bbb B}^2.$ Then $f$ has a homeomorphic
extension $f:\overline{{\Bbb B}^2}\rightarrow \overline{{\Bbb
B}^2}.$ Moreover, for any $z_0\in \overline{{\Bbb B}^2}$ there is
$C>0$ and $0<\widetilde{\varepsilon_0}<\varepsilon_0$ depending only
on $z_0$ and $Q$ such that (\ref{eq12E}) holds.}
\end{corollary}

\medskip
The proof of Corollary~\ref{cor2} almost literally repeats the proof
of Corollary~\ref{cor1}, and therefore is omitted.

\medskip
\begin{example}\label{ex5}
We have already constructed an example of a map satisfying the
conditions and the conclusion of Theorem~\ref{th3} for $n=3$. For
comparison, we will also construct a similar example on the plane.
Let $\beta$ be the function defined in~(\ref{eq2E}). Now put
$Q(z)=\beta(|z|).$ Now, by~(\ref{eq12D*}) we obtain that
$q_{0}(r)=q^{\,*}_{0}(r)=\beta(r).$ Set $\varepsilon_0:=1/4$ and
$0<a<1.$ Arguing similarly to~(\ref{eq2A}), we obtain that
$\int\limits_{a}^{\varepsilon_0} \frac{dr}{rq_{0}(r)}\geqslant
\sum\limits_{k=4}^{k_0-1}\int\limits_{\frac{1}{k+1}}^{\frac{1}{k}-2^{-4k-1}}\,\frac{dr}{r}.$
Now, by~(\ref{eq5F}) $\int\limits_{a}^{\varepsilon_0}
\frac{dr}{rq_{0}(r)}\geqslant \ln\frac{k_0}{4}-c.$ Thus, similarly
to~(\ref{eq3U}) and~(\ref{eq3V}),
\begin{equation}\label{eq3A} \lim\limits_{a\rightarrow +0}
\frac{\exp\left\{-\int\limits_{a}^{\varepsilon_0}
\frac{dr}{rq_{0}(r)}\right\}}{a}\leqslant C<\infty
\end{equation}
for some $0<C<\infty.$ Thus, the condition~(\ref{eq11A}) holds for
$\alpha=1$ at $0.$

According to Proposition~6.15 in~\cite{RSY$_1$}, the mapping
$$w=f(z)=\frac{z}{|z|}e^{\int\limits_{1}^{|z|}(\beta(t)/t)\,dt}$$
is a homeomorphism of the Sobolev class $W_{\rm loc}^{1, 1}({\Bbb
C})$ and has a dilatation $K_I(z, f)$ equal to $\beta(|z|).$ Note
that this map takes the disk ${\Bbb B}^2$ onto itself and moreover,
$f|_{{\Bbb S}^1}\equiv z.$ Moreover,
$$\int\limits_{{\Bbb B}^2} K_I(z, f)\,dm(z)=2\pi\int\limits_0^1 r\cdot \beta(r)\,dr=$$
\begin{equation}\label{eq3W}=2\pi\sum\limits_{k=1}^{\infty}
\int\limits_{\frac{1}{k+1}}^{\frac{1}{k}-2^{-4k-1}}\,r\,dr+
2\pi\sum\limits_{k=1}^{\infty}\int\limits_{\frac{1}{k}-2^{-4k-1}}^{\frac{1}{k}}
r\cdot 2^{k-1}\,dr\leqslant
\end{equation}
$$\leqslant \pi+ 2\pi\sum\limits_{k=1}^{\infty}2^{-4k-1}\cdot 2^{k-1}<\infty\,.$$
Reasoning similarly to~(\ref{eq29})--(\ref{eq310}), we may obtain
from~(\ref{eq3W}) that
\begin{equation}\label{eq4B}
\int\limits_{{\Bbb B}^2} \, \Vert
f^{\,\prime}(z)\Vert\,dm(z)<\infty\,.
\end{equation}
By~(\ref{eq4B}), $f\in W^{1, 1}({\Bbb B}^2).$
Thus, the mapping $f$ satisfies all the conditions of
Theorem~\ref{th1}, and the conclusion of this theorem at the point
$z_0=0$ is applicable for this mapping with $\alpha=1$.
\end{example}

\section{On H\"{o}lder  continuity of harmonic  mappings on the
unit ball in $\mathbb{R}^n$}\label{SharHold}

In this section, we provide initial results related to harmonic functions   and plan to publish further results in a future article.

In order to discuss the subject we first need a few basic definitions and results.
\begin{definition}\label{DfHar}
Let  $U$ be   an open subset of  $\mathbb{R}^n$.
A harmonic function (real valued) is a twice continuously differentiable function  $f:U \rightarrow \mathbb{R}$ that satisfies Laplace's equation, that is,
$${\Delta f:= \displaystyle {\frac {\partial ^{2}f}{\partial x_{1}^{2}}}+{\frac {\partial ^{2}f}{\partial x_{2}^{2}}}+\cdots +{\frac {\partial ^{2}f}{\partial x_{n}^{2}}}=0} $$ 
everywhere on $U$. In physics  notations  often we write  ${\displaystyle \nabla ^{2}f}$ instead of  $\Delta f$  and  this is usually written as
${\displaystyle \nabla ^{2}f=0}$.
A function   $f=(f_1,f_2,...,f_m): U\rightarrow \mathbb{R}^m$  is called  vector valued harmonic function  if  $f_i$, $i=1,2,...,m$,  are  real valued  harmonic functions.
\end{definition}
In  two  dimensions  harmonic  functions  form  a  useful,  strictly  larger  class of  functions  including  holomorphic  functions.

For  example,  harmonic  functions  still  enjoy  a mean-value property, as holomorphic functions do: \\
The mean value property:
If $B(x, r)$   is a ball with center $x$ and radius $r$  which is completely contained in the open set $G \subset  \mathbb{R}^n$, then the value $u(x)$ of a harmonic function $u:U \rightarrow \mathbb{R}$ at the center of the ball is given by the average value of u on the surface of the ball; this average value is also equal to the average value of $u$ in the interior of the ball.

To get an orientation what we can expect  concerning H\"{o}lder   continuity  of functions in   Orlicz-Sobolev classes   (and  the place of these classes   with respect to  Sobolev classes)  it seems useful to have in mind the following classical result.

\subsection{Morrey's  theorem}   Let  $G\subset\mathbb{R}^n$  be a
bounded open set with  $C^1$  boundary.  Assume   $n < p <\infty$
and  set  $\alpha= 1- n/p >0$. Then  every  function  $f\in W^{1,p}$
coincides  a.e.  with  a function  $\tilde{f} \in C^{0,\alpha}(G)$.
Moreover, there exists a constant  $C$  such that
$$|\tilde{f}|_{C^{0,\alpha}}\leqslant C |f|_{W^{1,p}} \quad {\rm for\,\, all}  \quad  f\in W^{1,p}(G),$$
where $|f|_{W^{1,p}}$  is  Sobolev norm of $f$   on $G$.
In statement of Morrey's  theorem  it is supposed that   $W^{1,p}$, $p>n$, i.e.,    $(h-1)$ (see below  Proposition \ref{PrSob1}) holds  for particular choice of  $\varphi(t)=t^p$, $p>n$.
Recall  some relations  between Sobolev  and  Orlicz-Sobolev space. If  $p>n$  and   $\liminf_{t\rightarrow +\infty} \varphi(t)/ t^p >c>0 $, $G$ bounded  then  $W^{1,\varphi} \subset W^{1,p}$.
If in addition  $G\subset\mathbb{R}^n$  is  a bounded open set with  $C^1$  boundary  and  $f\in W^{1,\varphi}(G)$, then  $f$  is $\alpha$-H\"{o}lder  on  $G$,  where     $\alpha= 1- n/p >0$.
By this in mind,  it seems natural question to  consider  what is a right  version of  Morrey's  theorem  for  Orlicz-Sobolev space ?

We start  with  the following  proposition  which  determines  the places of  Orlicz-Sobolev classes with respect to  Sobolev classes.

\begin{proposition}\label{PrSob1}{\sl\,
Introduce the hypothesis
{\rm (h1)}:  {\sl\,}
Let $n\geqslant 3,$  $\alpha\in (0, 1],$ and
let $\varphi: (0, \infty)\rightarrow [0, \infty)$ be a
non-decreasing Lebesgue measurable function

{\rm (h2)}:

%
\begin{equation}\label{eq7}
\int\limits_{1}^{\infty}\left(\frac{t}{\varphi(t)}\right)^
{\frac{1}{n-2}}dt<\infty.
\end{equation}

Then:
\begin{itemize}
\item[{\rm (S1)}] $W^{1,p}(D)$, $p> n-1$,    is   in Orlicz-Sobolev  class for  $\varphi(t)=t^p$.

\item [{\rm(S2)}]  Suppose that $\varphi$  satisfies  (h1).  If  $p>n$  and   $\liminf_{t\rightarrow +\infty} \varphi(t)/ t^p >c>0 $, $G$ bounded  then  $W^{1,\varphi} \subset W^{1,p}$.
If in addition  $G\subset\mathbb{R}^n$  is  a bounded open set with  $C^1$  boundary  and  $f\in W^{1,\varphi}(G)$, then  $f$  is $\alpha$-Holder  on  $G$,  where     $\alpha= 1- n/p >0$.

\item [{\rm(S3)}] Suppose that $\varphi$  satisfies  (h1) and (h2)   and  $D$ is a bounded domain in
$\mathbb{R}^n$, $n\geqslant 3 $,   and that  $f\in
W^{1,\varphi}(D).$  Then $W^{1,\varphi}(D)\subset W^{1,n-1}(D)$.
\end{itemize}}
\end{proposition}

\medskip
\begin{proof} (i) is readable. Note that  (ii) follows from Morrey's  theorem. Let us prove (iii).

Set $ A(t)=\left(\frac{t}{\varphi(t)}\right)^ {\frac{1}{n-2}}$.
Since integral  $\int\limits_{1}^{\infty}A(u) du <\infty$, then
$\int\limits_{t}^{2t}A(u) du=o(1)$, when  $t\rightarrow \infty$.
Further  $A(t)\leqslant A(u)$  for $t\leqslant u$  and therefore  $t
A(t) \rightarrow 0$ and by elementary consideration  $\varphi(t)=
M(t) t^{n-1}$, where $M(t)\rightarrow \infty $  when   $t\rightarrow
\infty$. Hence  there is $t_0>1$   such that   $\varphi(t)\geqslant
t^{n-1}$ for $t\geqslant t_0$. Next we conclude $|\nabla f(x)|
^{n-1}\leqslant \varphi\left(|\nabla f(x)|\right) $  if $|\nabla
f(x)|\geqslant t_0$. Hence it is readable   $W^{1,\varphi}(D)\subset
W^{1,n-1}(D)$.~$\Box$
\end{proof}


\medskip
Our approach in this section  is based on the following
results   which  we call local spatial  version  of  Privalov
theorem   for harmonic functions   and  which has an independent
interest.

\medskip
\begin{theorem}\label{thmloc1}{\sl\,
Suppose  that $0 < \alpha<1$, $\,h\,$ is a Euclidean
harmonic mapping from    $\mathbb{B}^n$, and  continuous on  $\overline{{\Bbb B}^{n}}$.\\
Let $x_0\in \mathbb{S}^{n-1}$  and   $ |h(x)- h(x_0)| \leqslant M
|x-x_0|^\alpha$ for $x\in \mathbb{S}^{n-1}$. Then there is a
constant   $M_n$   such that $(1-r)^{1-\alpha} |h'(r x_0)|\leqslant
M_n$, $0\leqslant  r <1$.}
\end{theorem}

For some  global results  of this type see for example \cite{AKM$_2$}  and literature cite there.

As we mention,  recall that   we can get  some versions of the previous theorems in Sections 1-4 for harmonic maps which are  immediate corollary  of Theorem  \ref{thmloc1}:

\medskip
\begin{corollary}\label{thhar}{\sl\, If $f\in W^{1, \varphi}({\Bbb B}^n)$ satisfies  the  condition of
Theorem~\ref{th4} and it   is in addition harmonic in the sense of  Definition \ref{DfHar}, then $f$ is
$\alpha$-H\"{o}llder on ${\Bbb B}^n$.}
\end{corollary}

\medskip
\begin{theorem}{\sl\, The following statements are true:
\begin{itemize}
\item[{\rm (i)}]   If under the  condition in statement
of   Theorem~\ref{th3} (with respect  to  point $x_0$) in addition $f$ is harmonic, then there is a constant   $M_n$   such that

(I-1)   $(1-r)^{1-\alpha} |f'(r x_0)|\leqslant M_n$, $0\leqslant  r
<1$.
\item[{\rm (ii)}] If under conditions
of  Theorem  \ref{th1} (with respect to point $z_0$)  in addition $f$ is harmonic, then (I-1) holds for $n=2$ with $z_0$ instead of $x_0$.

\item[{\rm (iii)}]  If under  the  conditions
of  Corollary  \ref{cor1} in addition $f$ is harmonic,  then $f$ is
$\alpha$-H\"{o}lder on  $\overline{{\Bbb B}^{n}}$.


\item[{\rm (iv)}]  Under  the  conditions
of  Corollary  \ref{cor2},  in addition $f$ is harmonic, then    $f$
is  $\alpha$-H\"{o}lder on  $\overline{{\Bbb B}^{2}}$.
\end{itemize}
}
\end{theorem}

\medskip
In addition  we prove   Propositions \ref{Pk1} and  \ref{Pk2}.

\subsection{Propositions \ref{Pk1} and  \ref{Pk2}}

Set  $Q_\varepsilon(x_0)$  the mean value of $Q$ over ball  $B(x_0,\varepsilon)$  and  $Q^+(x_0)$ supremum  of  $Q_\varepsilon(x_0)$ over  $0 < \varepsilon<\varepsilon_0$.

Kalaj and the first  author study mappings in plane and space which satisfy  the Poisson differential inequality:

(h3)  $|\triangle u|\leqslant a|\nabla u|^2+b$.

We start with  the planar case which  has  some very   specific
properties with respect to spatial case. Note  that   the subject of
harmonic quasiconformal (shortly hqc) mappings has been intensively
studied by the participants of the Belgrade Analysis Seminar (see
for example~\cite{BM} and  \cite{M$_3$} for more details)
\footnote{M. Pavlovi\'c, V. Markovi\'c, D. Kalaj, V. Bo\v zin,
M.~Arsenovi\'c, M. Markovi\'c, N. Laki\'c, D.  \v Sari\'c, M.~Kne\v
zevi\'c, V. Todor\v cevi\'c, M. Laudanovi\'c, M. Svetlik, I.~Ani\'c,
etc,... .} in particular by Kalaj, who  proved that if  $h$ is a hqc
mapping of the unit disk onto  a  Lyapunov domain, then $h$ is
Lipschitz (see~\cite{Ka$_1$}). Recently in~\cite{BM} it is proved
$h$ is co-Lipschitz.


Hence

\medskip
{\bf Proposition  A}.
{\it  Suppose $h:\mathbb{U}\rightarrow D$ is a hqc homeomorphism, where
$D$ is a  Lyapunov domain with $C^{1,\mu}$ boundary. Then $h$ is bi-Lipschitz {\rm (}shortly bi-Lip{\rm)}.}

\medskip
Here  we  prove for example the following:

\medskip
\begin{proposition}\label{Pk1}{\sl\,
Let $D$ be a Lyapunov  planar  domain  and  $f= g+
 \overline{h}$
injective  harmonic presenting orientation of $\mathbb{D}$ onto $D$
{\rm(}or more generally $C^2$ homeomorphism which satisfiies {\rm(h3))}    and
either {\rm(i):}  (\ref{eq1AE})   holds  a.e.  on $\mathbb{S}$  with
$Q=K_\mu$  and  constant $C=K_1$  or  {\rm(ii):}  $q_z(r)\leqslant c_1$
for  almost all $z\in \mathbb{D}$.
Then $f$ is bi-Lip.}
 \end{proposition}

Our proof is based  on  the  Hardy spaces  theory.

\medskip
\begin{proof} Since  $f$  injective  harmonic presenting orientation of $\mathbb{D}$ onto
$D$,  $|\mu(z)|< 1$  for  $z\in \mathbb{D}$, where     $\mu=\mu_f$.
Let  $E$  be the set of points $z_0$ on $\mathbb{S}$ for which the
finite  radial limit  $\mu^* $ exists.   From the  Hardy spaces
theory it is known   then the radial limit  $\mu^* $ exist a.e.  on
$\mathbb{S}$.

Suppose that (i) holds and  denote by $E_0\subset \mathbb{S}$ the set on which (i) holds.
Note that  the set $\mathbb{S} \setminus (E \cap E_0)$  is of measure $0$.

Let  $z_0\in E\cap E_0$   and $D(z_0,r)$  is the intersection of disk $B(z_0,r)$ with  $\mathbb{D}$.
Let $Q(r,z_0)$ be the  mean vale of $Q$ over $B(z_0,r)$.
Since the angular   limit $K_\mu^*(z_0)$  of $K_\mu$ also exists at $z_0$
  we conclude  that $Q(r,z_0)$ tends to  $K_\mu^*(z_0)/2$  from   (1.7)  that  $K_\mu^*(z_0)/2\leqslant K_1$.

Hence if we set  $K=2K_1$ and  $k= \frac{K-1}{K+1}$,  we conclude
that  $|\mu^*| \leqslant k$  a.e.  on $\mathbb{S}$.
 Next since   $|\mu|=|h'/g'|$   and  $h'/g'$ is  holomorphic function then  from Hardy spaces  theory
 $|\mu|\leqslant k$
on
 $\mathbb{D}$. Hence $f$ is K-qc  and therefore by {\bf Proposition  A}  bi-Lip  on
 $\mathbb{D}$.

 If  we suppose that (ii) holds on the set $E_1\subset \mathbb{S}$  and let  $\gamma_r$  be  the part of the circle $|z-z_0|=r$ in
$\mathbb{D}$. If $z_0\in E\cap E_1$ then as the above  we  first
conclude that $q_{z_0}(r)$ tends  to $K_\mu^*(z_0)/2$ and from (ii)
that  $K_\mu^*(z_0)/2\leqslant c_1$.
\end{proof}~$\Box$

\medskip
Now we will consider  a spatial version of  Proposition \ref{Pk1}.
We first need some  simple properties.
Suppose  that     $\,h\,$ is    a vector Euclidean harmonic mapping from
$\mathbb{B}^n$ into  $\mathbb{B}_M$.
Then
\begin{itemize} \label{Lhar0}
\item[(i)] If   $h(0)=0$, then

(1)  $|h(x)| \leqslant M_1 |x|$.

\item[(ii)]   $\,h\,$ is  Lipschitz  on any  ball $B(0,r_0)$, $r_0\in[0,1)$.

\item[(iii)]  In particular if, $h$ is harmonic function  on some domain $G\subset \mathbb{R}^n$, $h$   is locally
Lipschitz at every point $x_0\in G$.
\end{itemize}

(i)  It is clear that there is $M>0$    such  that    $|h(x)|
\leqslant 2 M  |x|$  for  $|x|\geqslant 1/2$ . Next  on
$\mathbb{B}_{1/2}$  partial derivatives of $h$  are bounded  and
therefore  (1) follows.

(ii)  partial derivatives of $h$  are bounded   $B(0,r_0)$ and hence  (ii) follows.

(iii)  follows from (ii).

The following lemma shows that the above properties hold in more general setting:
\medskip
\begin{lemma}{\sl\,
Let   $f: U \rightarrow \mathbb{R}^m$ be   a differentiable
function {\rm(}where  $U \subset  \mathbb{R}^n$ is open{\rm)} and  $F$ compact
subset of $U$  and suppose   that  {\rm($h6$)}: partial derivatives are
bounded on $F$ {\rm(}in particular {\rm(}$h6${\rm)} is satisfied if  $f$ is $C^1$ on
$U${\rm)}. Then $f$ is Lip on $F$.}
\end{lemma}

\medskip
\begin{proof}
 Note that  $d={\rm dist}\,(F,\partial U)>0$. Let   $x,y \in F$
and set $h=y-x$.  If  $|h|\leqslant d$, then by  the  Mean value
theorem in several variables one finds points $x + t_i h$ on the
line segment $[x,y]$  satisfying ${\displaystyle
f_{i}(x+h)-f_{i}(x)=\nabla f_{i}(x+t_{i}h)\cdot h.}$ By  hypothesis
(h6) there is a constant $M>0$   such that  $|\nabla f_{i}
(x)|\leqslant M$  on $F$.  Hence $f_i$ is Lip on $F$ if   $|y-x
|\leqslant d$.

Suppose now  that   $|y-x |\geqslant d$. Note first since $F$ is
compact  and  $f$ is continuous  on $U$, then there is a constant
$M_1>0$   such that  $|f|$  is bounded on $F$. Next the ratio    of
$|fy-fx |$  and $|y-x |$  is bounded by  $2M_1 d^{-1}$  and we
conclude  that therefore  $f_i$  is Lip on $F$  and therefore  $f$
is Lip on $F$.
\end{proof}~$\Box$

\medskip In order to formulate a spatial version of Proposition
\ref{Pk1}  we need  some definitions and a result.

For  $x_0\in \mathbb{R}^n$ and  $0\leqslant r_1<r_2$  we define
$A(x_0,r_1,r_2)=\{x: r_1< |x-x_0|<r_2 \}$ which call a  spherical
ring. If  $x_0=0$  we write simply $A(r_1,r_2)$.

Let $G$  be an open subset of  $\mathbb{R}^n$  and  $f: G\rightarrow
\mathbb{R}^n$. We say that $f$  has finite distortion if, first of
all  if  $f\in W^{1,1}_{\rm loc}(G,\mathbb{R}^n)$  and there is a
function  $K(x)=K(x,f),1\leqslant  K(x)<\infty$, defined a.e. in
$G$   such that

(i)  $\Vert f^{\,\prime}(x)\Vert^{n}\leqslant K(x)J(x, f)$ a.e. $G$.

The smallest  $K=K_f$  satisfying (i) is called the outer dilatation function of $f$.

Kalaj and the first  author study mappings in plane and space which satisfy  the Poisson differential inequality:

(h3)  $|\triangle u|\leqslant a|\nabla u|^2+b$.

In~\cite{Ka$_3$}  Kalaj  proved  (see also   subsection \ref{ssFr1},Further
results)  for a more general result)

\medskip
{\bf Theorem K}.   {\it A quasiconformal mapping of the unit ball onto a domain with $C^2$ smooth boundary,
satisfying the Poisson differential inequality, is Lipschitz continuous}.

\medskip
\begin{proposition} \label{Pk2}{\sl\, Let $G\subset \mathbb{R}^n$  be $C^2$ domain.
If $f: {\Bbb B}^n\xrightarrow{onto} G$  harmonic homeomorphism {\rm(}or
more generally $C^2$ homeomorphism which satisfies {\rm($h3$))} and there
are  a $r_0\in (0,1)$ and non negative  function  $Q$   defined
a.e. on  the  ring  $A(r_0,1)$ such that  $K(x, f)\leqslant Q(x)$
for a.e. $x\in {\Bbb B}^n$   and {\rm(h4):}   $Q^{+}$ is bounded  on ring
$A(r_0,1)$, then $f$ is Lipschitz  on  ${\Bbb B}^n$.}
\end{proposition}

\medskip
\begin{proof} The  hypothesis  {\rm(h4)}  implies that  there is $K\geqslant 1$   such that  $K_f\leqslant K$  on  the  ring  $A(r_0,1)$. Hence by Theorem K,  $f$ is Lipschitz  on $A(r_0,1)$.
Since by Lemma \ref{Lhar0},  $f$ is Lipschitz on $B(0,r_0)$ the result follows.
\end{proof}~$\Box$

\subsection{Local spatial  version  of  Privalov theorem   for harmonic functions}

First we need some definitions and  properties of  spherical cap.
Recall  by $\omega_{n-1}$ we denote  the  surface of
$n-1$-dimensional sphere $\mathbb{S}^{n-1}$  (pay attention that
some authors prefer notation $\mathbb{S}^{n}$  for
$(n-1)$-dimensional sphere and   $\omega_{n}$ for its area).  Then
the surface $(n-1)$-dimensional measure  of sphere $S(0, r)$ of
radius $r$ is $P(r)=\omega_{n-1}r^{n-1}$.

\medskip
\begin{definition}(Spherical-polar    cap).
We can define  the spherical  cap  in terms of the so-called contact
angle (the angle between the normal to the sphere at the bottom of
the cap and the base plane). More precisely,  we use the following
notations  $\hat{x}=x/|x|$ and  $\hat{0}=e_1$,
$S(\hat{x},\gamma)=\{y\in \mathbb{S}_n: \langle y,\hat{x}\rangle
\geqslant \cos \gamma \}$ for  the polar cap with center  $\hat{x}$,
where    $\gamma$  is the spherical angle of it. In a similar way in
planar case  we define   $C(\hat{x},\gamma)=\{y\in \mathbb{S}_1:
\langle y,\hat{x}\rangle \geqslant \cos \gamma \}$.
\end{definition}
\medskip

If $f$ is  a function on  $\mathbb{S}^{n-1}$ which is constant on
$\partial\mathbb{S}^\varphi=\{t \in \mathbb{S}^{n-1}: t_n = \cos
\varphi \}$ for  $0\leqslant \varphi \leqslant \pi$  we say that $f$
depends only on $\varphi$. Let  $0\leqslant \varphi \leqslant \pi$;
the surface of  spherical (polar)  cap  $\mathbb{S}^\varphi$ of
radius $\varphi$   is
$$ A(\varphi)=\omega_{n-1}\int\limits_0^{\sin\varphi} \frac{r^{n-2}}{x_n}dr \,.$$
By change of variables $r=\sin  \theta,\, dr= \cos \theta\,
d\theta\,\,,   x_n=\cos \theta$ in the previous formula,   we find
\begin{equation}\label{fCup1a}
A(\varphi)=\omega_{n-1}\int\limits_0^\varphi \sin^{n-2}\theta
\,d\theta \,.
\end{equation}
See      also proof of  Theorem 2 and formula (11)  in Section V  in
L. Ahlfors  book \cite{Ahl$_3$}.

\medskip
\begin{proposition}\label{f.cup0}{\sl\,
If  $f$  is a  function on $\mathbb{S}^{n-1}$ which depends only
on $\varphi$,
then\\
$\int\limits_{\mathbb{S}^{n-1}}f d\sigma=  \int\limits_0^\pi
f(\varphi)\,A'(\varphi)\,d\varphi= \omega_{n-1}\int\limits_0^\pi
f(\varphi)\,\sin^{n-2} (\varphi)\,d\varphi$.}
\end{proposition}

\medskip
We only outline a proof. Let  $ 0=\varphi_0 <\varphi_1 < \varphi_2 <
\dots \varphi_n =\pi$, $ \varphi_{k-1} < \xi_k < \varphi_k$ and

$S_n=\sum\limits_{k=1}^n   f(\xi_k)
\big(A(\varphi_k)-A(\varphi_{k-1})\big)$. Then  $S_n\rightarrow
\int\limits_{\mathbb{S}^{n-1}}f d \sigma$  and  $S_n\rightarrow
\int\limits_0^\pi f(\varphi)\,A'(\varphi)\,d\varphi $  when
$n\rightarrow \infty$. Hence since $A'(\varphi)=\omega_{n-1}
\sin^{n-2}\varphi$ the proof follows.

\section{Proof of Theorem~\ref{thmloc1}}\label{SharHoldB}
 Before to proceed to the
proof we need a few definition and some elementary propererties  of
spatial harmonic functions.
\begin{definition}
Recall  if  $x,y\in \mathbb{R}^n$  by $|x-y|$ we denote  Eucledian distance between $x$ and $y$.

Further in this definition,  we  suppose  that  (i) $G$ is a domain  in   $\mathbb{R}^n$   and  $f:G \rightarrow
\mathbb{R}^m$.
\begin{enumerate}
\item  We say   $f$  is locally  H\"{o}lder  ($\alpha$ -H\"{o}lder) at
$x_0 \in G$  if

$Lf_\alpha (x_0)= \limsup\limits_{G\ni x\rightarrow x_0} |f(x)-
f(x_0)/|x-x_0|^\alpha < \infty$  when   $G\ni x \rightarrow x_0$.

If   $\alpha =1$ we say that  $f$   is locally   Lipschitz at $x_0$. We write $Lf (x_0)$  instead  $Lf_\alpha (x_0)$.
We can adapt the above definition for  $f: \mathbb{S}^{n-1}\rightarrow \mathbb{R}^m$.

We say that $f$  is H\"{o}lder continuous, when there are nonnegative real constants $C, 0 <\alpha \leq 1$, such that
$${\displaystyle |f(x)-f(y)|\leq C |x-y|^{\alpha}} $$
for all $x$ and $y$ in the domain  $G$  of $f$.  The number $\alpha$ is called the exponent of the H\"{o}lder condition  and we also say $f$ is   $\alpha$-H\"{o}lder continuous  and write  $f\in {\rm Lip}_\alpha(G,\mathbb{R}^m)$. By   $|f|_{C^{0, \alpha}}$  we denote the smallest constant for which the pervious inequality holds
and call H\"{o}lder norm of $f$ on $G$.

\item If $\alpha= 1$  in the previous inequality, then we say  that  the function satisfies a Lipschitz condition  or it is  Lipschitz continuous (shortly Lip) on $G$  with multiplicative constant $C$.
  If $m=n$  in (i)   and  $f$  is homeomorphism  and  both $f$ and $f^{-1}$ are Lipschitz    we say that $f$ is bi-Lipschitz  (shortly bi-Lip).

\item   If the function $f$ and its derivatives up to order $k\in \mathbb{N}$ are bounded on the closure of $G$  and  $\alpha$-H\"{o}lder continuous, then we say that  $f$ belongs to the H\"{o}lder space ${\displaystyle C^{k,\alpha}({\overline {G}})}$.

\item  Let  $\gamma$  be a closed rectifiable Jordan planar curve  of length $s_0$,   $G$ domain  enclosed by   $\gamma$ and
$s$ an  arc length parametar  on $\gamma$, and  $\gamma_0(s)$,  $s\in [0,s_0]$, arc length parametarization of  $\gamma$. We say that $G$ is  $C^{k,\alpha}$ domain if  $\gamma_0$
is  $C^{k,\alpha}$ on  $[0,s_0]$.
In the literature  planar domain $D$ is  called  Lyapunov  domain  if $D$   has  smooth $C^{1,\alpha}$- boundary  for some $0<\alpha<1$.
\end{enumerate}
\end{definition}

In the proof of next theorem we use  the representation of  harmonic functions (see  formula (\ref{RepHar1}) below)  by means  the  Poisson kernel for harmonic functions  on  the unit ball $\mathbb{B}^n$ which is  given by
$$P(x,\eta)=\frac{1-|x|^2}{\omega_{n-1} |x-\eta|^n}\,,$$
where $x\in\mathbb{B}^n$  and  $\eta\in\mathbb{S}^{n-1}$,   and
positive Borel measure $d\sigma$   on $\mathbb{S}^{n-1}$.
By   $d\sigma$ we denote
positive Borel measure on $\mathbb{S}^{n-1}$ invariant with respect  to orthogonal    
group $O(n)$  normalized such that   $\sigma(S^{n-1})=1$.

\medskip
First  recall the statement of  Theorem  \ref{thmloc1}.

\medskip
\begin{theorem}\label{thmloc0}{\sl\,
Suppose  that $0 < \alpha<1$,    $\,h\,$ is    a Euclidean harmonic
mapping from $\mathbb{B}^n$ which is
continuous on  $\overline{{\Bbb B}^{n}}$ , and \\
(h1)  Let $x_0\in \mathbb{S}^{n-1}$  and   $ |h(x)-h(x_0) |
\leqslant M |x-x_0|^\alpha$ for $x\in \mathbb{S}^{n-1}$.

Then there is a constant   $M_n$   such that

$(1-r)^{1-\alpha} |h'(r x_0)|\leqslant M_n$, $0\leqslant  r <1$. }
\end{theorem}

\medskip
\begin{proof}
Let  $h_b$ denote  the restriction of $h$ on   $\mathbb{S}^{n-1}$.
Since   $h$ is harmonic on $\mathbb{B}^n$  and continuous on
$\overline{{\Bbb B}^{n}}$, then
\begin{equation}\label{RepHar1}
h(x)=\int\limits_{\mathbb{S}^{n-1}} P(x,\eta) h_b(\eta)
d\sigma(\eta)
\end{equation}
for every $x\in \mathbb{B}^n$. Set  $d:= 1  - |x|^2$.  By
computation $\partial_{x_k} P(x,t)= -(\frac{2 x_k}{|x-t|^n} + d
n\frac{x_k-t_k}{|x-t|^{n+2}})$. Hence  if  $d\leqslant |x-t|$,  then

(1)  $ |\partial_{x_k} P(x,t)|\leqslant c_1 \frac{1}{|x-t|^n}$.

Let  $x=r e_n$  and  $\theta$ the angle between  $t$  and $e_n$. Then
$s:= |x-t|^2= 1 -2r \cos \theta + r^2$ depends only on $\theta$ for
fixed  $x$.
Next  since  $\int\limits_{\mathbb{S}^{n-1}} \partial_k P(x,t)
h(e_n) d\sigma(t)=0$, we find
\begin{equation}\label{}
\partial_{x_k}h(x) = \int\limits_{\mathbb{S}^{n-1}} \partial_k P(x,t) \big(h(t) -h(e_n)\big)  d\sigma(t)\,.
\end{equation}
Hence by  (1) and the hypothesis  $(h1)$, we get
\begin{equation}\label{est.loc0}
|\partial_{x_k}h(x)| \leqslant c_2 \int\limits_{\mathbb{S}^{n-1}}
\frac{|e_n -t|^\alpha}{|x-t|^n} d\sigma(t)\,.
\end{equation}
Therefore  the   proof of Theorem  \ref{thmloc0} is reduced to the proof of the following proposition.
\end{proof}~$\Box$

\medskip
\begin{proposition}\label{prop2a}{\sl\, Suppose  that $0 <\alpha<1$  and $x=re_n$, $0<r<1$. Then
$$I_\alpha(r e_n)=: \int\limits_{\mathbb{S}^{n-1}} \frac{|e_n -t|^\alpha}{|x-t|^n}
d\sigma(t)\leqslant c\cdot \frac 1{(1-r)^{1-\alpha}},
$$}
where   $c=c(\alpha,n)$  is a positive constant which depends only on  $n$ and $\alpha$.
\end{proposition}

\medskip Using similar  approach  if  $\omega$ is a majorant one
can prove

$$I_\omega(r e_n)=: \int\limits_{\mathbb{S}^{n-1}} \frac{\omega(|e_n -t|)}{|x-t|^n}
d\sigma(t)\leqslant c\cdot \frac {\omega(\delta_r)} {\delta_r}.
$$

\medskip
\begin{proof}
We use        spherical cups    $S^{\theta}$  defined  by  $t_n  >
\cos \theta$  and  integration with parts. Since for a fixed
$\theta\in [0,\pi]$, $|e_n -t|\leqslant \theta$  for  $t\in
S^{\theta}$, by  an application of    Proposition \ref{f.cup0} to
$f(t)=\frac{|e_n -t|^\alpha}{|x-t|^n}$, we get (see also Remark
\ref{Rm1} below)
\begin{eqnarray}\label{eqH4}
I_\alpha(r e_n) \leqslant c_3 \int\limits_0^\pi \frac
{|\theta|^{n-2}
|\theta|^{\alpha}} {((1-r)^2+\frac{4r}{\pi^2}\theta^2)^{n/2}}\,d\theta < \\
\label{eqH5} c_4\int\limits_0^{\infty}\frac {\theta^{\alpha +n-2}}
{\left((1-r)^2+\frac{4r}{\pi^2}\, \theta^2\right)^{n/2}}\,d\theta\,.
\end{eqnarray}
Next using  $(1+\frac{4r}{\pi^2} u^2)^{-1} \leqslant c_5
(1+u^2)^{-1}$  for  $\frac 12 \leqslant r <1$  and  the  change of
variable  $\theta=(1-r) u$, we find
\begin{equation}
I_\alpha(r e_n) \leqslant c_6 {(1-r)^{\alpha -1
}}\int\limits_0^{\infty}\frac{u^{\alpha +n-2}}{(1+u^2)^{n/2}}\,d
u\,.
\end{equation}
Denote by $J(\alpha)$   the last expression  on the right hand side  of previous formula.
Hence  since  $g(u)= \frac{u^{\alpha +n-2}}{(1+u^2)^{n/2}}\sim u^{\alpha-2}$ for   $u\rightarrow  + \infty$ and  by  hypothesis   $0 <\alpha<1$  and therefore   $\alpha-2 <-1$,  the integral $J(\alpha)$  converges  and
and therefore

(i)  $I_\alpha(r e_n) \leqslant c_7 (1-r)^{\alpha -1}$  for  $\frac
12 \leqslant r <1$.


$(1-r)^{1-\alpha} A(r)$ is continuous on  $[0,1/2]$ and attains a maximum $c_8$, that is

(ii)    $I_\alpha(r e_n) \leqslant c_9 (1-r)^{\alpha -1}$  for  $0
\leqslant r \leqslant  \frac 12 $, where  $c_9= c_3c_8 $.

Hence  from (i) and (ii)  with $c=\max \{c_7,c_9 \}$  the proof of Proposition    follows.
\end{proof}~$\Box$

\medskip
Combining Proposition~\ref{prop2a} and (\ref{est.loc0}) we get proof
of Theorem.

\medskip
\begin{remark}\label{Rm1}
It is convenient  to denote  expressions  by  $A(r)$  and
$B(r)$ that appear  on the right-hand side  in formula~(\ref{eqH4}) and
(\ref{eqH5})  without constants $c_3$  and $c_4$  respectively. Note
that  $A(0)$ is finite and that  $B(0)= +\infty$. In order to
estimate  $A(r)$  we use  the  change of variable  $\theta=(1-r) u$
and therefore the integral $A(r)$ can be  transformed  to integral
over   $[0,a(r)]$  with respect to $u$, where  $a(r)=\pi
(1-r)^{-1}$.  Since  $a(r)\rightarrow \infty$ if  $r\rightarrow  1$,
it is convenient to
estimate integral $A(r)$  by integral  $B(r)$  over interval  $[0,\infty)$.
\end{remark}

\medskip
\begin{remark}
Instead of  (ii)  we can based the proof of Theorem  on the
following inequality:
$$|\partial_{x_k}h(x)|  \leqslant c_7 \cdot \frac 1{(1-r)^{1-\alpha}}$$  for $\frac 12 \leqslant |x|<1 .$
Hence since on  $\mathbb{B}_{1/2}$  partial derivatives are bounded  readably  proof of   Theorem \ref{thmloc0}  follows.
Note that the above proof breaks down for  $\alpha=1$ because  $J(1)= \infty$.
Moreover,  for each  $n=2$, there is a Lipschitz continuous map  $f:\mathbb{S}^{n-1}\rightarrow \mathbb{R}^n$   such that  $u=P[f]$  is not Lipschitz continuous.
In planar case, consider $f=u+iv$  such that  $z f'=- \log (1-z)$. $u'_\theta$ is bounded while its harmonic conjugate $ru'_r$ is not bounded.
In spatial case, consider $U(x_1,x_2,... x_n)= u(x_1+ix_2,x_3,... x_n)$.
\end{remark}

\medskip
\subsection{Further results}\label{ssFr1}
Using an approach as  in~\cite{AKM$_1$}, we can prove further
results. Here we only  announce  the following results:

\medskip
\begin{theorem}{\sl\,
Suppose  that $f:\mathbb{S}^{n-1}\rightarrow \mathbb{R}^n$  is
locally Lipschitz {\rm (Lip-$1$)}   at $x_0 \in \mathbb{S}$,  $f\in
L^\infty(\mathbb{S}^{n-1})$   and $\,h\,=P[f]$ is a Euclidean
harmonic mapping from
$\mathbb{B}^n$.\\
Then
\begin{itemize}
\item[{\rm S4)}]  $$|h'(r x_0)T|\leqslant M $$

\noindent for every  $0\leqslant r <1$ and unit vector  $T$ which is
tangent on $\mathbb{S}^{n-1}_r$  at  $r x_0$,  where $M$ depends
only on $n$, $|f|_\infty$  and  $Lf (x_0)$.

If we suppose in addition that  $h$ is  K-quasiregular (shortly K-qr)  mapping  along  $[o,x_0)$,  then\\
\item[{\rm S5)}]   $$|h'(r x_0)|\leqslant K \,M$$
for every  $0\leqslant r <1$.
\end{itemize}}
\end{theorem}

We can extend  our results to   class of moduli functions which
include $\omega(\delta)=\delta^{\alpha}~(0< \alpha \leqslant 1)$, so
our result generalizes earlier results on H\"{o}lder continuity (see
\cite{NO}) and Lipschitz continuity (see~\cite{AKM$_1$}).

In addition,  concerning further research  we suggest some possibility.
Suppose that domains $D$ and  $\Omega$ are bounded domains in
$\mathbb{R}^n$  and its boundaries belong to class  $C^{k,\alpha},
0\leqslant \alpha \leqslant 1,k\geqslant 2$ (more generally  $C^2$).
Suppose further that $g$ and $g'$ are  $C^1$  metric on
$\overline{D}$  and $\overline{\Omega}$  respectively.  Using inner
estimate (cf. Theorem 6.14 \cite{GT}[10]) we can prove

{\bf Theorem MM} (Theorem 6.9 \cite{M$_2$}). {\it  If $u:D\rightarrow
\Omega$ a qc {\rm($g,g'$)}-harmonic map {\rm(}or satisfies {\rm($h3$))}, then $u$ is
Lipschitz on $D$.}

We discussed this result at Workshop on Harmonic Mappings and Hyperbolic Metrics,  Chennai,  India, Dec. 10-19, 2009, and in  \cite{M$_2$}, where a proof is outlined.
For more details see \cite{mm.Spring2021}.

We now present a few open questions.

Using {\bf Theorem MM} or   {\bf Theorem K} we can prove.

{\bf Theorem B}.   {\it Let $G\subset \mathbb{R}^n$  be $C^2$ domain. If $f: {\Bbb B}^n\xrightarrow{onto} G$  harmonic homeomorphism  and there is a $r_0\in (0,1)$ such that  $(h4)$:   $Q^{+}$ is bounded  on ring  $A(r_0,1)$, then $f$ is Lipschitz  on  ${\Bbb B}^n$.}
\bques
If we suppose instead of {\rm($h4$)} only {\rm($h5$)}:  $Q^{+}$ is bounded  on $\mathbb{S}^{n-1}$, whether  $f$ is Lipschitz  on  ${\Bbb B}^n$?

What is right version of  Theorem  1.1  and   4.1   if in addition it is supposed that     $f$ is  harmonic?
\eques


\section{On Bi-Lipschitz  qc maps}\label{biLip}

Recall  that the condition (\ref{eq11B}) provides sufficient conditions for H\"{o}lder and Lipschitz continuity.
In this section we  show  that  in some situation if Beltrami   coefficient is H\"{o}lder continuous that the map is bi-Lip.
In order to discuss the subject we first need some preliminaries.
\begin{definition}
\begin{enumerate}
\item
Let $f$ be  a  complex valued  function defined an open set planar set $V$.
We use notation  $z=x+iy$ for complex  numbers and  for complex partial derivatives  $f_z:= (f_x -if_y)/2$  and   $f_{\overline{z}}:= (f_x + if_y)/2$,  where $f_x$  and  $f_y$   are partial derivatives with respect
coordinates $x$ and $y$.  In the literature  frequently  notation $\partial f$ and  $\overline{\partial}f$  are used instead  $f_z$ and $f_{\overline{z}}$  respectively.

In this section let $\Omega$  denote   a planar domain.
\item     An equation
\begin{equation}\label{belt.eq}
\overline{\partial}f  =  \mu \, \partial f ,
\end{equation}
where  $\mu$  is a  complex valued measurable function defined a.e. on $\Omega$ and  $||\mu||_\infty  <  1$  is essential supremum with respect to $L^\infty$-norm,  is
called $\mu$-{\it Beltrami equation}  on $\Omega$.

\item   If a  homeomorphism  $f: \Omega \xrightarrow{onto}\Omega_*\subset \mathbb{C}$ satisfies
\begin{itemize}
\item[{\rm (i)}] $f$  is  $ACL$ on  $\Omega$, and
\item[{\rm (ii)}]   $ |f_{\overline{z}} | \leqslant
k\, |f_z|$  almost everywhere in  $\Omega$, where  $k =
\frac{K-1}{K +1}\in [0,1)$,\\
we say that $f$ is a  $K$- quasiconformal  (shortly qc); more precisely  $K$-qc in analytic sense.
\end{itemize}
\end{enumerate}
\end{definition}
The last item 3. of the  definition is equivalent to requirement that:\\
(A)  $f$ is  homeomorphism   and it has locally
integrable distributional derivatives which satisfy  (ii).

\medskip
\begin{theorem}{\bf(Existence theorem)}{\sl\,
Let $\mu$ be a measurable function in a domain $\Omega$  with  $
||\mu||_\infty < 1$.  Then there is a qc mapping of $\Omega$ whose
complex dilatation agrees with $\mu$ a.e.}
\end{theorem}

\medskip
In this setting we say that  $f$ is solution of  Beltrami equation
for $\mu$. The complex dilatation at $z_{0}$ is
\begin{equation}
\mu _{f} = \frac{f_{\overline{z}}}{f_{z}}.
\end{equation}

\medskip
Frequently  the notation  ${\rm Belt}(f)$   is also  used instead of  $\mu _{f}$.

\medskip
\begin{theorem}\label{Stoilowqc}{\sl\,
Let $f$ and $g$ be qc map of a domain $\Omega$ whose complex
dilatations agree a.e. in $\Omega$. Then $f\circ g^{-1}$ is a
conformal mapping.}
\end{theorem}

\medskip
In order to get a  feeling  of the subject  we  first consider some examples.

Let  $f_0$   be a branch  of  $\sqrt{z}$. Then  $\mu_{f_0}=0$   and  $f_0$  has  a singularity at $0$ and $\mu_g=0$  a.e. on $\mathbb{C}$. Next, by Theorem \ref{Stoilowqc},
$g$ is a M\"{o}bius transformation $A= f_0\circ g^{-1}$ is a  conformal mapping. Thus  $f_0=A\circ g$ is  a  conformal mapping. More precisely  if $G$ is a simple connected domain
which does not contain $0$,
$A\circ g$  is  a  conformal mapping on $G$.

The following known  example shows that  qc with continuous Beltrami   coefficient are not  $C^1$  in general.

\begin{example}\cite{AIM, BGR}\label{Ex1}
Consider  $f(z)=-z \ln|z|^2$   for   $|z|\leqslant r_0=e^{-2}$. Then
$f: B(0,r_0) \rightarrow B(0,4r_0)$  and

$$f_{\bar z}=-\frac{z}{\overline{z}},     f_z=   -1 - \log|z|^2, \mu_f=  \frac{z}{\bar z (1 + \log|z|^2)}  .$$

Hence  it  is qc with continuous Beltrami   coefficient, and $f_{\bar z}$  and  $f_z$ are discontinuous at $0$, and therefore
yet  $f$ is not $C^1$.
\end{example}

Note that    in planar case  we frequently use notation   $\mathbb{S}$    instead  of   $\mathbb{S}^1$  and     $\mathbb{U}$    instead  of   $\mathbb{B}^2$.

We can modify this example to show that there is  $f\in QC(\mathbb{U})$  such that  $\mu_f$  is continuous on $\mathbb{U}$, but  $f_{\bar z}$  and  $f_z$ are discontinuous at
some point $z_0\in \mathbb{U}$.
We will show that if in addition the second dilatation  $\nu_f$  is anti holomorphic on $\mathbb{U}$  and
$f$  is  $C^{1}$  up to the boundary, then $f$  is biLipschitz   and $\mu_f$   continuous  up to the boundary.
\begin{example}\label{Ex2}
Consider   $f_0(z)= \frac{z}{\log|z|^2}$. We check that  $(\log|z|^2)_z=1/z$,  $p= \frac{1}{\log|z|^2}- \frac{z}{(\log|z|^2)^2} /z = \frac{1}{\log|z|^2} A(z) $, where  $A(z)= 1-\frac{1}{\log|z|^2}$.  Next $q= -\frac{z}{(\log|z|^2)^2} 1/\overline{z}=- \frac{z}{\overline{z}}\frac{1}{(\log|z|^2)^2}$    and therefore  $\mu_{f_0}=-\frac{z}{\overline{z}}\frac{1}{\log|z|^2} B(z)$, where $B=1/A$.
For $r_0$  small enough  $f_0$ is qc on $B(0,r_0)$,  $p(0)=q(0)=\mu_f(0)=0$, $\mu_{f_0}$ and $\nu_{f_0}$ are   continuous   $f_0$ is $C^1$, but there is no finite $(f_0^{-1})_x $ at $0$.  Next if   $B=B(z_0,s_0)$  is an arbitrary planar disk
using  the mapping  $f(z)= f_0(\lambda (z-z_0))$  with $\lambda s_0=r_0 $,  we conclude
that
there is  a qc $C^1$ map $f$  on $B$ such that   $\mu_f$ and $\nu_f$ are   continuous on $B$, but $f^{-1}$  has no  finite derivatives at  $w_0=f(z_0)$.
But note that  if  both  $\mu_f$ and $\nu_f$ are    $\alpha$ H\"{o}lder continuous on $B$, then  $f$ and  $f^{-1}$ are  $C^{1,\alpha}$.
\end{example}

The following example shows that   the Beltrami coefficient $\mu_f$  of a qc $f$ is uniformly  $\alpha$ -H\"{o}lder   (and therefore $f$ is  $C^{1,\alpha}$ up to the boundary) but  it does not imply  in general  that $f^{-1}$    has  continuous extension.
\begin{example}\label{Ex3} Let  $0<\alpha< \beta<1$, $\gamma=\beta-\alpha$  and   $0<k<1$.
Solve equation  $f_z=(1-z)^\alpha$  and  $f_{\overline{z}}=k (1-\overline{z})^\alpha$.  Then  check that   (i):  $f$ is  $C^{1,\alpha}$  up to the boundary of the unit disk, but    $\mu_f$  is discontinuous at $1$  if $\alpha= \beta$  and $\mu_f$  is  $\gamma$ -H\"{o}lder  on  unit disk $\overline{\mathbb{U}}$  if $\alpha< \beta$.
We can write   $f=g\circ T$, where
$T(z)=1-z$,   and   $g(w)= \frac{w^{\alpha+1}}{\alpha+1} + k\frac{\overline{w}^{\beta+1}}{\beta+1}$.

Now consider $g$ on  $B=B(0,1)$. We are going to show  that $g$ has corresponding properties from which (i) follows.
Check that  $\mu_g=k\frac{\overline{w}^{\beta}}{w^{\alpha}}$ is
$\gamma$ -H\"{o}lder and   $|\mu_g|\leqslant k$  on the closed disk
$\overline{B}$.
Let  $\phi$  be  conformal mapping of  $G=g(B)$  onto  $B$  with  $\phi(0)=0$
and set    $h=\phi \circ g$.
Thus     $h$ is a qc mapping  which maps $B$ onto itself   such that  $\mu_h=\mu_g$  and $h(0)=0$.
In addition,   $\mu_h=\mu_g$  is  $\gamma$ -H\"{o}lder  on $\overline{B}$, but  (ii): partial derivatives  of $g^{-1}$     do not have continuous extension to $0$.
At this  point  it seems  natural    to check whether  partial derivatives of  $h^{-1}$  have  continuous extension to $h(0)=0$; we leave it  to the reader and note  that Theorem \ref{ThKal} below  shows that it is  the case.\\
Warning:   Note  here  that   $h(B)=B$  is smooth domain  and
$G=g(B)$  is not smooth (precisely  only  at a point $0$). Therefore  there is an essential difference between $g$ and $h$:  $g$  does not satisfy  the hypotheses of  Theorem \ref{ThKal} below  and $h$ does it.
\end{example}

Note further  that hypothesis  that  $\mu$ is a compactly supported
function in H\"{o}lder spaces completely  changes the situation.
Namely,  there is  a classical result that goes back to Schauder
which  asserts that  $f$   is of class $C^{1,\epsilon}$ provided $\mu$ is a compactly
supported function in  ${\rm Lip}_\varepsilon(\mathbb{C},\mathbb{C})$, stated here as (see, for example,  Theorem  2.10  and 2.12,   Ch II,  \$ 5,  p.93 in Vekua's book \cite{Vekua} and
\cite[Chapter~15]{AIM}):

{\bf Theorem S}.    If  $\mu$ is a complex valued compactly
supported  $\epsilon$- H\"{o}lder  continuous  function  on $\mathbb{C}$, $0<\epsilon<1$,  with  $|\mu|_\infty <1 $,   than   principal solution $f$  of $\mu$-Beltrami equation  is of class $C^{1,\epsilon}$.

In order  to discuss  some version  of  Kellogg   and  Warschawski theorem  for a class of quasiconformal maps  we first need some definition and results.

Recall  in the literature  planar domain $D$ is  called  Lyapunov  domain  if $D$   has  smooth $C^{1,\alpha}$- boundary  for some $0<\alpha<1$. We first  recall  the  classical   result   of   Kellogg   and  Warschawski  related to Riemann conformal mapping.



{\bf Kellogg's theorem}.   Let $\gamma$ be a Jordan curve. By the
Riemann mapping theorem there exists a Riemann conformal mapping of
the unit disk onto the Jordan domain $G = {\rm int} \gamma$. By
Caratheodory's theorem it has a continuous extension to the
boundary. Moreover, if $\gamma \in C^{n,\alpha}$, $n\in
\mathbb{N},\,\, 0\leqslant \alpha < 1$, then the Riemann conformal
mapping has a $C^{n,\alpha}$  extension to the boundary (this result is known as Kellogg's theorem).

In \cite{Ka$_2$}   Kalaj     gives   some   extensions   of   classical   results   of   Kellogg   and  Warschawski to a class of quasiconformal (q.c.) mappings. Among the other results the author  states   the following:
\begin{theorem}\label{ThKal}  Suppose that  ($H_1$):   $f$ is   a q.c. mapping  between two planar domains  $G$ and $G'$  with smooth $C^{1,\alpha}$ boundaries.
Then the following conditions are
equivalent:

{\rm (A)}  $f$  together with its inverse mapping $f^{-1}$,  is  $C^{1,\alpha}$  up to the boundary.

{\rm (B)}  the Beltrami coefficient $\mu_f$  is uniformly  $\alpha$ H\"{o}lder continuous {\rm ($0<\alpha <1$)}.
\end{theorem}

It our impression   that this result can be related  with some resent results.  For example,   we derive  a  small extension of this result (see Theorem \ref{ThMSS} below)  and  we also can get from this result
Theorem 1.3  \cite{APE}.


\medskip

It is also interesting that Theorem  \ref{ThKal} is related  to   a result  of
Mateu, Orobitg and Verdera  (Theorem MOV  below) proved  in
\cite{MOV}:

{\bf Theorem MOV}.
{\it Principal solution of Beltrami equation with  H\"{o}lder continuous Beltrami  coefficient supported on a  Lyapunov  domain is bi-Lipschitz.}

Using  this  result    we can    fill  a small gap  in  original  proof in~\cite{Ka$_2$}  and  prove   (A) iff (B).
Now we are ready  to prove  the following:
\begin{theorem}\label{ThMSS}{\sl\,
Let  $f$ be   a q.c. mapping  between two planar domains  $G$ and
$G'$  with smooth $C^{1,\alpha}$, $0<\alpha <1$, boundaries. Then
the following conditions are equivalent:

{\rm (A)}  $f$  together with its inverse mapping $f^{-1}$,  is  $C^{1,\alpha}$  up to the boundary.

{\rm (B)}  the Beltrami coefficient $\mu_f$  is uniformly  $\alpha$ H\"{o}lder continuous {\rm ($0<\alpha <1$)}.

{\rm (C)}    $f=\phi\circ f_0$,  where   $\phi$ is conformal mapping
from  Lyapunov $f_0(G)$  onto $G'$  and $f_0$  is bi-Lipschitz.}
\end{theorem}

\medskip
\begin{proof} Suppose (B). We  first  prove that (B) implies (C).  Using  Kellogg's theorem
without loss  of  generality  we can reduce the proof to  the case
$G=\mathbb{B}^2$. In this setting for given  $r_0>1$   there is
H\"{o}lder continuous  $\mu_0$ supported on $B=B(0,r_0)$, $r_0>1$,
such that   $\mu_0=\mu_f$ on  $\mathbb{B}^2$. Namely extend
$\mu_f$ to  $\mu$   by reflection $\mu(z)= \mu_f(Jz)$, where
$Jz=1/\overline{z}$. Next let  $\varphi \in C^2_0(B)$,  $\varphi =1$
on $\overline{{\Bbb B}^{2}}$  and  set  $\mu_0=\varphi \mu$.
If  $f_0$ is   principal solution of $\mu_0$-Beltrami equation, then
we have  $f=\phi\circ f_0$, where   $\phi$ is conformal mapping from
$G_0:=f_0(\mathbb{B}^2)$  onto $G'$. Since  $\mu_0$  is H\"{o}lder
continuous on $\overline{{\Bbb B}^{2}}$, by  Theorem S we conclude
that  $p_0:=(f_0)_z$  and   $q_0:=(f_0)_{\overline{z}}$ are
H\"{o}lder continuous on $\overline{{\Bbb B}^{2}}$  and by Theorem
MOV  that  $f_0$  is bi-Lipschitz.   Hence  there are $0<l_0 <L_0$
such that  $l_0 \leqslant  |p_0|- |q_0| $   and  $|p_0|+ |q_0|
\leqslant L_0$ on $\overline{{\Bbb B}^{2}}$. Next  by abusing of
notation write $f_0'(t)$ instead of  $(f_0)_{b}'(t)$,  where
$(f_0)_{b}(t)=f_0(e^{it})$, $0\leqslant t \leqslant  2 \pi$. If
$\gamma (t)=f_0(e^{it})$, $0\leqslant t \leqslant  2 \pi$ and $s$ an
arc length parametar  on $\gamma$, then
$\gamma'(s)=f_0'(t)/|f_0'(t)|$, where $f_0'(t)=(f_0)_{b}'(t)$. Since
$f_0'(t)= i( p_0 e^{it}- q_0 e^{-it}) $,   $l_0 \leqslant |f_0'(t)|$
and therefore  $\gamma'(s)$ is H\"{o}lder continuous  on $[0,s_0]$,
where  $s_0$   is length of curve $\gamma$. Therefore  we have
proved (C).

Now we prove  that  (C) implies (A). By  Kellogg's theorem $\phi$
and $\phi^{-1}$ have a continuous extension to  $\overline{G_0}$ and
$\overline{G'}$ respectively and therefore  $\phi$   is
bi-Lipschitz. Hence  from (C) it follows that  $f$ is  bi-Lipschitz.

Recall that we suppose that $G=\mathbb{B}^2$.
Next $p:=f_z$  is   $\alpha$- H\"{o}lder  on $\overline{{\Bbb B}^{2}}$  and  there is $m_0 >0$
such that  $|p|\geqslant m_0 $  on $\overline{{\Bbb B}^{2}}$   and   since
$$1/p(z_1) -  1/p(z_2)
= \frac{p(z_2)- p(z_1)}{p(z_1)p(z_2)},$$ we get
$$|1/p(z_1) -
1/p(z_2)|\leqslant C |z_2-z_1|^\alpha/m^2_0$$  and therefore  $1/p$
is $\alpha$- H\"{o}lder  on $\overline{{\Bbb B}^{2}}$. Hence also the second dilatation
$\nu_f=\mu_f \frac{p}{\overline{p}}$  is  $\alpha$- H\"{o}lder  up
to the boundary  of  $G=\mathbb{B}^2$.
Finally, since  $f$ is bi-Lipschitz   and
$\mu_{f^{-1}}= -\nu_f \circ f^{-1}$, we conclude that  $\mu_{f^{-1}}$ is    $\alpha$- H\"{o}lder  up to the boundary.  Hence (A) follows.
\end{proof}~$\Box$
For further  discussion we first need  the following
definition: if $f:G\rightarrow G'$  is differentiable  at  $z_0$
and  $J_f(z_0) \neq 0$  we say  that $f$  is regular at $z_0$.
\begin{remark}\label{rm1}

After writing final version of this manuscript,  Kalaj    turned our attention on the following results.
\begin{theorem}  [Theorem 7.1 \cite{LV-73}, p. 232]\label{ThReg}
Let $G$ and $G'$ be domains in $\mathbb{C}$  and  $w:G\rightarrow G'$  a qc  with  complex dilatation  $\chi$, where  $|\chi(z)|\leq k< 1 $  a.e. in $G$. If there is $\chi_0$
such  that
\be \label{regCon}
\int\limits_{B(z_0,r_0)} \frac{|\chi(z)- \chi_0|}{|z-z_0|^2} dA< + \infty
\ee
for some $r_0 >0$, then  $w$ is regular at $z_0$  and $\chi(z_0)=\chi_0$.
\end{theorem}

If   $\chi$  satisfies  the condition   (\ref{regCon})  we say that   $\chi$  satisfies   the   integral growth  condition  at $z_0$, and   if
$\chi$  satisfies   the   integral growth  condition at every point  of $G$  we say that  $\chi$  satisfies   the   integral growth  condition  on $G$.

It is interesting that this result infer   simple proof  of a few results  including,   Theorem MOV,and Astala,  Prats, and  Saksman  Theorem 1.3  \cite{APE} stated here as

{\bf Theorem ASP}.  Let   $0 < s< 1$, let  $G$
 be a simply connected, bounded  $C^{1,s}$-domain and let  $g:G\rightarrow G$
be a $\mu$-quasiconformal mapping, with $supp (\mu) \subset \overline{G}$
and $\mu \in C^s(G)$. Then $g\in C^{1,s}(G)$.

Further from Theorem LK  we can also infer the following:
\begin{proposition}\label{PrMOV}
If  $f:\mathbb{C}\rightarrow \mathbb{C}$   is a $C^1$  qc  with  complex dilatation  $\chi$, where  $|\chi(z)|\leq k< 1 $  a.e. in $\mathbb{C}$,   $f(\infty)=\infty$, and
$\chi$  satisfies   the   integral growth  condition  at  $\mathbb{C}$,
then $f$ is Bi-Lip  on every compact  subset  of  $\mathbb{C}$.

If in addition  $supp \chi$  is bounded set then    $f$ is Bi-Lip   on  $\mathbb{C}$.
\end{proposition}

We can consider Proposition \ref{PrMOV} as a generalization of  Theorem MOV.

\end{remark}

{\it Acknowledge .}  We are indebted to professor Kalaj   who turned our attention on Theorem \ref{ThReg} and  an anonymous referee for careful reading the  manuscript and useful suggestion which improved exposition.

\medskip

\medskip
{\bf \noindent Miodrag Mateljevic} \\ University of Belgrade,
Faculty
of Mathematics \\
16 Studentski trg, P.O. Box 550\\
11001 Belgrade, SERBIA\\
miodrag@matf.bg.ac.rs

\medskip
{\bf \noindent Ruslan Salimov} \\
Institute of Mathematics of the NAS of
Ukraine,  \\
3, Tereschenkivska st., \\ 01024 Kiev-4,  UKRAINE,
\\
ruslan.salimov1@gmail.com

\medskip
{\bf \noindent Evgeny Sevost'yanov} \\
{\bf 1.} Zhytomyr Ivan Franko State University,  \\
40 Bol'shaya Berdichevskaya Str., \\ 10 008  Zhytomyr, UKRAINE \\
{\bf 2.} Institute of Applied Mathematics and Mechanics\\
of NAS of Ukraine, \\
1 Dobrovol'skogo Str., \\ 84 100 Slavyansk,  UKRAINE\\
esevostyanov2009@gmail.com


\begin{thebibliography}{99}

\bibitem[ARS]{ARS} {\sc Afanasieva, E., V.~Ryazanov. and R.~Salimov:}
{\it On mappings in the Orlicz-Sobolev classes on Riemannian
manifolds}. - Journal of Mathematical Sciences 181:1, 2012, 1--17.


\bibitem[Ahl$_1$]{Ahl$_1$} {\sc Ahlfors, L.:} {\it On quasiconformal mappings.} - J. Analyse Math., 3
1954, 1--58.

\bibitem[Ahl$_2$]{Ahl$_2$} {\sc Ahlfors, L.:} {\it M\"{o}bius transformation in several
dimensions.} - School of mathematics, University of Minnesota, 1981.

\bibitem[Ahl$_3$]{Ahl$_3$} {\sc Ahlfors, L. V.:} {\it Lectures on Quasiconformal mappings
with additional chapters by Earle and Kra,} Univ. Lectures Series
(Providence , R.I.),  Shishikura, Hubbard, 2006.

\bibitem[AM$_1$]{AM$_1$}  {\sc Arsenovi\'c, M. and  V.~Manojlovi\'c:} {\it On
the modulus of continuity of harmonic quasiregular mappings on the
unit ball in $\mathbb{R}^n.$} - Filomat 23(3), 2009, 199-202.


\bibitem[AKM$_1$]{AKM$_1$} {\sc Arsenovi\' c, M., V.~Koji\'c and M.~Mateljevi\' c :} {\it On Lipschitz continuity of
harmonic quasiregular maps on the unit ball in $R^n.$} - Ann. Acad.
Sci. Fenn. 33, 2008, 315-318.

\bibitem[AKM$_2$]{AKM$_2$} {\sc Arsenovi\' c, M.,  V.~Koji\'c  and M.~Mateljevi\' c :} {\it Lipschitz-type spaces and
harmonic mappings in the space.} - Ann. Acad. Sci. Fenn. 35:2, 2010, 379-387.
\bibitem[AM]{am20}  {\sc Arsenovi\' c M.  and  M. Mateljevi\' c}, {\it On Ahlfors-Beurling Operator}, Ukr. Math. Bull    18 (2021), No 3, 292-302.

\bibitem[AIM]{AIM} {\sc Astala, K., T.~Iwaniec, and  G.J.~Martin:} {\it Elliptic partial
differential equations and quasiconformal mappings in the plane.} -
Princeton University Press, Princeton and Oxford, 2009.
\bibitem[APE]{APE}  Astala K.,  M. Prats,  E. Saksman,    {\it  Global smoothness of quasiconformal mappings in the Triebel-Lizorkin scale},  Submitted on 23 Jan 2019, arXiv:1901.07844v1
\bibitem[BGR]{BGR} {\sc  Bojarski, B.,  V.~Gutlyanskii,  O.~Martio, and  V.~Ryazanov:}
{\it Infinitesimal geometry of quasiconformal and bi-Lipschitz
mappings in the plane.} EMS Tracts in Mathematics, vol. 19. -
European Mathematical Society (EMS), Z\"{u}rich, 2013.


\bibitem[BM]{BM} {\sc Bo\v zin,  V.,  M. Mateljevi\'c :}
{\it Quasiconformal and HQC Mappings Between Lyapunov Jordan
Domains.} Accepted: 2018-12-31, pp. 23, DOI Number: 10.2422/2036-2145.201708$_-$013, Ann. Sc. Norm. Super. Pisa Cl. Sci.(5) Vol. XXI (2020),107-132

\bibitem[Fe]{Fe} {\sc Federer, H.:} {\it  Geometric Measure Theory.}
- Springer, Berlin etc., 1969.

\bibitem[Cal]{Cal} {\sc Calderon A.P.:} {\it On the differentiability of absolutely
continuous functions}. - Riv. Math. Univ. Parma 2, 1951, 203--213.

\bibitem[Cr]{Cr} {\sc Cristea, M.:} {\it  Open discrete mappings having local $ACL^n$
inverses.} - Complex Variables and Elliptic Equations 55:1--3, 2010,
61--90.

\bibitem[Onof]{Onof2018}  {\sc  D'Onofrio  L.:}     {\it Differentiability versus approximate differentiability},
Electronic Journal of Differential Equations, Conference 25 (2018), pp. 77–85.ISSN: 1072-6691. URL: http://ejde.math.txstate.edu or http://ejde.math.unt.edu

\bibitem[GT]{GT} {\sc Gilbarg, D.,  N.~Trudinger:}
{\it Elliptic partial Differential  Equation of Second Order.} --
Springer Verlag, Berlin, Heidelberg, 2001.

\bibitem[GG]{GG} {\sc Gutlyanski\u{i}~V.Ya. and ~Golberg A.:}
{\it  On Lipschitz continuity of quasiconformal mappings in space.} - J.
d'Anal. Math.~109, 2009, 233--251.


\bibitem[HK]{HK} {\sc Hencl, S. and P.~Koskela:}
{\it Lectures on mappings of finite distortion.} - Lecture Notes in
Mathematics 2096, Springer, 2014.

\bibitem[IM]{IM} {\sc  Iwaniec~T. and ~Martin G.:}
Geometrical Function Theory and Non-Linear Analysis. - Claren\-don
Press, Oxford, 2001.

\bibitem[IS$_1$]{IS$_1$} {\sc Il'yutko~D.P. and E.A.~Sevost'yanov:}  {\it Open discrete
mappings with unbounded coefficient of quasi-conformality on
Riemannian manifolds.} - Sbornik Mathematics 207:4, 2016, 537--580.

\bibitem[IS$_2$]{IS$_2$} {\sc Il'yutko~D.P. and E.A.~Sevost'yanov:} {\it Boundary behaviour
of open discrete mappings on Riemannian manifolds.} - Sbornik
Mathematics 209:5, 2018, 605--651.


\bibitem[Ka$_1$]{Ka$_1$} {\sc  Kalaj, D.:}  {\it Quasiconformal harmonic mapping
between Jordan domains.} -  Math. Z. 260:2, 2008, 237-252.

\bibitem[Ka$_2$]{Ka$_2$} {\sc Kalaj, D.:} {\it  On Kellogg's theorem for quasiconformal mappings.} - Glasgow
Mathematical Journal 54:3, 2012, 599-603.

\bibitem[Ka$_3$]{Ka$_3$} {\sc Kalaj, D.:} {\it  A priori estimate of gradient of a solution to certain
differential inequality and quasiregular mappings.} - Journal
d'Anal. Math. 119:1, 2013, 63-88.

\bibitem[Ka$_4$]{CorrGlasg}
 D. Kalaj,   {\it Corrigendum  to the  paper David Kalaj: On Kellogg's theorem for quasiconformal mappings}, Glasgow Mathematical Journal /Volume 54 /Issue 03/ September 2012, pp 599-603,
 Glasgow Math. J.: 2020. doi:10.1017/S001708952000062

\bibitem[KM]{KM} {\sc Kalaj, D. and M.~Mateljevi\'c:} {\it
Inner estimate and quasiconformal harmonic maps between smooth
domains.} - Journal d'Anal. Math. 100, 2006, 117-132.

\bibitem[KR]{KR} {\sc Kovtonyuk, D. and V.~Ryazanov:} {\it  New modulus estimates
in Orlicz-Sobolev classes.} - Annals of the University of Bucharest
(mathematical series) LXIII: 5, 2014, 131--135.

\bibitem[KRSS]{KRSS} {\sc Kovtonyuk, D., V.~Ryazanov, R.~Salimov and E.~Sevost'yanov:} {\it  Toward the theory
of Orlicz-Sobolev classes.} - St. Petersburg Math. J. 25:6, 2014,
929--963.

\bibitem[LV]{LV-73} {\sc Lehto, O. and K.~Virtanen:} {\it  Quasiconformal Mappings
in the Plane.} -  Springer: New York etc., 1973.

\bibitem[LSS]{LSS} {\sc Lomako, T.,  R.~Salimov  and E.~Sevostyanov :}
On equicontinuity of solutions to the Beltrami equations. -  Annals
of the University of Bucharest (mathematical series) LIX:2, 2010,
263--274.

\bibitem[MV]{MV}
{\sc Mateljevic, M. and M. Vuorinen:} {\it On harmonic
quasiconformal quasi-isometries.} - J. Inequal., 2010:1, 2010, 1-19.

\bibitem[M$_1$]{M$_1$} {\sc Mateljevi\' c,  M.:}
{\it Distortion of quasiregular mappings and equivalent norms on
Lipschitz-type spaces.} -- Abstract and Applied Analysis 2014, 2013,
1-20.

\bibitem[M$_2$]{M$_2$} {\sc Mateljevi\'c, M.:}
{\it The Lower Bound for the Modulus of the Derivatives and Jacobian
of Harmonic Injective Mappings.} -  Filomat 29:2, 2015, 221-244.

\bibitem[M$_3$]{M$_3$} {\sc  Mateljevi\' c,  M.:}
{\it Invertible harmonic and Harmonic quasiconformal mappings.} -
Filomat 29:9, 2015, 1953-1967.
\bibitem[M$_4$]{mm.Spring2021} {\sc  Mateljevi\'c, M.:},    {\it  Boundary behaviour of  partial derivatives for  solutions to    certain Laplacian-gradient inequalities   and  spatial qc maps},
accepted for publication in Springer Proceedings in Mathematics \&
Statistics (ISSN: 2194-1009, https://www.springer.com/series/10533 ), vol.357, Operator Theory and Harmonic Analysis.
Vol. 1: New General Trends and Advances of the Theory 2021, 393-418
\bibitem[MRV]{MRV} {\sc Martio, O.,  S.~Rickman and  J.~V\"{a}is\"{a}l\"{a}:}
{\it  Distortion and singularities of quasiregular mappings.} - Ann.
Acad. Sci. Fenn. Ser. A1 465, 1970, 1--13.

\bibitem[Mar]{Mar} {\sc Martio, O.:} {\it On harmonic quasiconformal mappings}, Ann.
Acad. Sci. Fenn., Ser. A I 425 (1968), 3-10.


\bibitem[MRSY$_1$]{MRSY$_1$} {\sc Martio, O., V.~Ryazanov, U.~Srebro, and
E.~Yakubov:} {\it  Moduli in Modern Mapping Theory}. -- Springer Science +
Business Media, LLC: New York, 2009.

\bibitem[MRSY$_2$]{MRSY$_2$} {\sc Martio, O., V.~Ryazanov, U.~Srebro, and
E.~Yakubov:} {\it  On $Q$-ho\-me\-o\-mor\-phisms.} - Ann. Acad. Sci.
Fenn. Math. 30:1, 2005, 49--69.

\bibitem[MOV]{MOV} {\sc Mateu, J.,  J.~Orobitg, and  J.~Verdera:}
{\it Extra cancellation of even Calder\'{o}n-Zygmund operators and
quasiconformal mappings.} --  J.  Math.  Pures Appl.~91: 4, 2009,
402-431.

\bibitem[Ma]{Ma} {\sc Maz'ya, V.:} {\it Sobolev classes.} - Springer, Berlin--New York, 1985.

\bibitem[M]{M} {\sc Mori, A.:} {\it  On an absolute constant in the theory of quasiconformal mappings.} - J. Math.
Soc. Japan 8, 1956, 156--166.


\bibitem[NO]{NO} {\sc Nodler, C.A. and  D.M.~Oberlin:} {\it Moduli of continuity and a Hardy-Littlewood
theorem.} - Lecture Notes in Math. 1351, p.~265-272,
Springer-Verlag, Berlin etc., 1988.


\bibitem[RSS]{RSS} {\sc Ryazanov, V., R.~Salimov and E.~Sevost'yanov:} {\it  On
the H\"{o}lder property of mappings in domains and on boundaries}. -
Ukr. Mat. Visnyk 16:3, 2019, 383-402 (in Russian).

\bibitem[RSY$_1$]{RSY$_1$} {\sc Ryazanov, V., U. Srebro and E. Yakubov:} {\it Plane
mappings with dilatation dominated by func\-tions of bounded mean
oscillation.} - Sib. Adv. in Math. 11:2, 2001, 94--130.

\bibitem[RSY$_2$]{RSY$_2$} {\sc Ryazanov, V., U. Srebro and E. Yakubov:} {\it On ring solutions of Beltrami
equations.} - J. d'Anal. Math. 96, 2005, 117--150.


\bibitem[RS]{RS} {\sc Ryazanov, V. and E.~Sevost'yanov:} {\it  Toward the theory of ring
$Q$-homeomorphisms.}  - Israel J. Math. 168, 2008, 101--118.

\bibitem[Re]{Re} {\sc Reshetnyak, Yu.G.:}
Space mappings with bounded distortion. - Transl. Math. Monographs 73, AMS, 1989.

\bibitem[Sa]{Sa} {\sc Saks, S.:} {\it Theory of the Integral.} - Dover, New York, 1964.

\bibitem[Sal]{Sal} {\sc Salimov, R.R.:} {\it On regular homeomorphisms in the
plane.} - Annales Academi{\ae} Scientiarum Fennic{\ae} Mathematica 35, 2010, 285-–289.


\bibitem[Sev$_1$]{Sev$_1$} {\sc Sevost'yanov, E.A.:} {\it On the boundary behavior of
open discrete mappings with unbounded characteristic.} - Ukrainian
Math. J. 64:6, 2012, 979--984.

\bibitem[Sev$_2$]{Sev$_2$} {\sc Sevost'yanov, E.A.:}
{\it Towards a theory of removable singularities for maps with unbounded
characteristic of quasi-conformity.} - Izv. Math. 74:1, 2010,
151--165.

\bibitem[Sev$_3$]{Sev$_3$} {\sc  Sevost'yanov, E.:}
{\it Investigation of space mappings by geometric method.} -  Kiev:
Naukova Dumka, 2014, (in Russian).

\bibitem[Sm]{Sm} {\sc Smolovaya, E.S.:}  {\it Boundary behavior of ring $Q$-homeomorphisms in metric
spaces.} - Ukr. Mat. Zh., 62:5, 2010, 682--689 (in Russian); English
transl. in Ukr. Math. Journ., 62:5, 2010, 785--793.

\bibitem[Suv]{Suv} {\sc Suvorov, G.D.:}
{\it Generalized principle of length and area in mapping theory.} -
Naukova Dumka: Kiev, 1985.

\bibitem[Va$_1$]{Va$_1$} {\sc V\"{a}is\"{a}l\"{a}, J.:}
{\it Two new characterizations for quasiconformality.} - Ann. Acad. Sci.
Fenn. Ser. A 1 Math. 362, 1965, 1--12.

\bibitem[Va$_2$]{Va$_2$} {\sc V\"{a}is\"{a}l\"{a}, J.:}
{\it  Lectures on $n$-dimensional quasiconformal mappings.} - Lecture Notes
in Math. 229, Springer-Verlag, Berlin etc., 1971.

\bibitem[Vec]{Vekua} {\sc  Vekua I.N.:}  {\it Generalized analytic functions}.(Russian).2nd ed., revised, Nauka, Moskow, 1988, 509 p;
I. N. Vekua,  Generalized Analytic Functions Paperback – April 10, 2014,
Ian Sneddon (Translator).

\bibitem[Zy]{Zy} {\sc Zygmund,  A.:} {\it Trigonometrical Series,} $2^{\rm {nd}}$ ed., Vol. I,
II. - Cambridge university Press, New York, 1959.




\end{thebibliography}
\end{document}